\documentclass[12pt]{article}
\usepackage{MyStyleSheet}
\usepackage{hyperref}

\usepackage{stmaryrd}
\usepackage{blindtext}
\usepackage{geometry}
\usepackage{tikz}
\usepackage{theoremref}
\usepackage{wasysym}
\usepackage{harpoon}
\usepackage{float}
\usepackage{caption}
\usepackage[title]{appendix}

\def\graphScale{0.29}
\def\largeGraphScale{0.5}
\def\specialVertexColor{white}
\def\diagramScale{0.8}
\def\sqrt3{1.732}
\def\vertexWidth{0.1}

\newcommand{\maxEdgeDiagram}{
\begin{tikzpicture}[scale=\diagramScale]
    \draw (1, 0) -- (0, 1);
    \draw (1.5, 0.5) -- (1, 0);
    \draw (1.5, 0.5) -- (1, 1);
    \draw (0,1) -- (1,1);
    \draw (0,0) -- (1,0);
    \draw (0,0) -- (0,1);
    \draw (0,0) -- (1,1);
    \draw (1,0) -- (1,1);
    \draw (1.5 + \sqrt3, 0.5 + 1) -- (1.5, 0.5);
    \draw (1.5 + \sqrt3, 0.5 + 1) -- (1.5 + 1.5 * \sqrt3, 0.5 + 0.5);
    \draw (1.5 + \sqrt3, 0.5 + 1) -- (1.5 + \sqrt3, 0.5 - 2);
    \draw (1.5, 0.5 - 1) -- (1.5, 0.5);
    \draw (1.5, 0.5 - 1) -- (1.5 + \sqrt3, 0.5 - 2);
    \draw (1.5, 0.5 - 1) -- (1.5 + 1.5 * \sqrt3, 0.5 + 0.5);
    \draw (1.5 + 1.5 * \sqrt3, 0.5 - 1.5) -- (1.5, 0.5);
    \draw (1.5 + 1.5 * \sqrt3, 0.5 - 1.5) -- (1.5 + \sqrt3, 0.5 - 2);
    \draw (1.5 + 1.5 * \sqrt3, 0.5 - 1.5) -- (1.5 + 1.5 * \sqrt3, 0.5 + 0.5);
    \draw[fill=black] (0, 0) circle (\vertexWidth);
    \draw[fill=black] (0, 1) circle (\vertexWidth);
    \draw[fill=black] (1, 0) circle (\vertexWidth);
    \draw[fill=black] (1, 1) circle (\vertexWidth);
    \draw[fill=\specialVertexColor] (1.5, 0.5) circle (\vertexWidth);
    \draw[fill=black] (1.5 + \sqrt3, 0.5 + 1) circle(\vertexWidth);
    \draw[fill=black] (1.5 + 1.5 * \sqrt3, 0.5 + 0.5) circle(\vertexWidth);
    \draw[fill=black] (1.5 + 1.5 * \sqrt3, 0.5 - 1.5) circle(\vertexWidth);
    \draw[fill=black] (1.5 + \sqrt3, 0.5 - 2) circle(\vertexWidth);
    \draw[fill=black] (1.5, 0.5 - 1) circle(\vertexWidth);
\end{tikzpicture}
}

\newcommand{\fourVertexCompleteGraph}{
\begin{tikzpicture}[scale=\graphScale]
    \draw (0, 0) -- (1, 0);
    \draw (0, 0) -- (1, 1);
    \draw (0, 0) -- (0, 1);
    \draw (1, 0) -- (1, 1);
    \draw (1, 0) -- (0, 1);
    \draw (0, 1) -- (1, 1);
    \draw[fill=black] (0, 0) circle (0.125);
    \draw[fill=black] (0, 1) circle (0.125);
    \draw[fill=black] (1, 0) circle (0.125);
    \draw[fill=black] (1, 1) circle (0.125);
\end{tikzpicture}
}
\newcommand{\hamSandwich}{
\begin{tikzpicture}[scale=\graphScale]
    \draw (0, 0) -- (1, 0);
    \draw (0, 0) -- (0, 1);
    \draw (1, 0) -- (1, 1);
    \draw (1, 0) -- (0, 1);
    \draw (0, 1) -- (1, 1);
    \draw[fill=black] (0, 0) circle (0.125);
    \draw[fill=black] (0, 1) circle (0.125);
    \draw[fill=black] (1, 0) circle (0.125);
    \draw[fill=black] (1, 1) circle (0.125);
\end{tikzpicture}
}

\newcommand{\trianglePlusIsolated}{
\begin{tikzpicture}[scale=\graphScale]
    \draw (0, 0) -- (1, 0);
    \draw (0, 0) -- (1, 1);
    \draw (1, 0) -- (1, 1);
    \draw[fill=black] (0, 0) circle (0.125);
    \draw[fill=black] (0, 1) circle (0.125);
    \draw[fill=black] (1, 0) circle (0.125);
    \draw[fill=black] (1, 1) circle (0.125);
\end{tikzpicture}
}
\newcommand{\SuzukiThirtyTwoElusiveGraph}{
\begin{tikzpicture}[scale=\graphScale]
    \draw (0, 0) -- (1, 0);
    \draw (0, 0) -- (0, 1);
    \draw (0, 0) -- (1, 1);
    \draw (1, 0) -- (1, 1);
    \draw[fill=black] (0, 0) circle (0.125);
    \draw[fill=black] (0, 1) circle (0.125);
    \draw[fill=black] (1, 0) circle (0.125);
    \draw[fill=black] (1, 1) circle (0.125);
\end{tikzpicture}
}

\newcommand{\house}{
\begin{tikzpicture}[scale=\largeGraphScale]
    \draw (0, 0) -- (1, 0);
    \draw (0, 0) -- (1, 1);
    \draw (0, 0) -- (0, 1);
    \draw (1, 0) -- (1, 1);
    \draw (1, 0) -- (0, 1);
    \draw (0, 1) -- (1, 1);
    \draw (1.5, 0.5) -- (1, 0);
    \draw (1.5, 0.5) -- (1, 1);
    \draw[fill=black] (0, 0) circle (0.125);
    \draw[fill=black] (0, 1) circle (0.125);
    \draw[fill=black] (1, 0) circle (0.125);
    \draw[fill=black] (1, 1) circle (0.125);
    \draw[fill=\specialVertexColor] (1.5, 0.5) circle (0.125);
\end{tikzpicture}
}
\newcommand{\balloon}{
\begin{tikzpicture}[scale=\largeGraphScale]
    \draw (0, 0) -- (1, 1);
    \draw (0, 0) -- (0, 1);
    \draw (0, 1) -- (1, 1);
    \draw (1.5, 0.5) -- (1, 0);
    \draw (1.5, 0.5) -- (1, 1);
    \draw[fill=black] (0, 0) circle (0.125);
    \draw[fill=black] (0, 1) circle (0.125);
    \draw[fill=black] (1, 0) circle (0.125);
    \draw[fill=black] (1, 1) circle (0.125);
    \draw[fill=\specialVertexColor] (1.5, 0.5) circle (0.125);
\end{tikzpicture}
}
\newcommand{\codart}{
\begin{tikzpicture}[scale=\largeGraphScale]
    \draw (0, 0) -- (1, 1);
    \draw (0, 0) -- (0, 1);
    \draw (0, 1) -- (1, 1);
    \draw (1.5, 0.5) -- (1, 1);
    \draw[fill=black] (0, 0) circle (0.125);
    \draw[fill=black] (0, 1) circle (0.125);
    \draw[fill=black] (1, 0) circle (0.125);
    \draw[fill=black] (1, 1) circle (0.125);
    \draw[fill=\specialVertexColor] (1.5, 0.5) circle (0.125);
\end{tikzpicture}
}
\newcommand{\northStar}{
\begin{tikzpicture}[scale=\largeGraphScale]
    \draw (0, 0) -- (1, 0);
    \draw (0, 0) -- (1, 1);
    \draw (1, 0) -- (1, 1);
    \draw (1.5, 0.5) -- (1, 0);
    \draw (1.5, 0.5) -- (1, 1);
    \draw[fill=black] (0, 0) circle (0.125);
    \draw[fill=black] (0, 1) circle (0.125);
    \draw[fill=black] (1, 0) circle (0.125);
    \draw[fill=black] (1, 1) circle (0.125);
    \draw[fill=\specialVertexColor] (1.5, 0.5) circle (0.125);
\end{tikzpicture}
}
\newcommand{\dart}{
\begin{tikzpicture}[scale=\largeGraphScale]
    \draw (0, 0) -- (1, 0);
    \draw (0, 0) -- (1, 1);
    \draw (1, 0) -- (1, 1);
    \draw (0, 1) -- (1, 1);
    \draw (1.5, 0.5) -- (1, 0);
    \draw (1.5, 0.5) -- (1, 1);
    \draw[fill=black] (0, 0) circle (0.125);
    \draw[fill=black] (0, 1) circle (0.125);
    \draw[fill=black] (1, 0) circle (0.125);
    \draw[fill=black] (1, 1) circle (0.125);
    \draw[fill=\specialVertexColor] (1.5, 0.5) circle (0.125);
\end{tikzpicture}
}
\newcommand{\happyFace}{
\begin{tikzpicture}[scale=\largeGraphScale]
    \draw (1, 0) -- (1, 1);
    \draw (1.5, 0.5) -- (1, 0);
    \draw (1.5, 0.5) -- (1, 1);
    \draw[fill=black] (0, 0) circle (0.125);
    \draw[fill=black] (0, 1) circle (0.125);
    \draw[fill=black] (1, 0) circle (0.125);
    \draw[fill=black] (1, 1) circle (0.125);
    \draw[fill=\specialVertexColor] (1.5, 0.5) circle (0.125);
\end{tikzpicture}
}

\newcommand{\misshapenhouse}{
\begin{tikzpicture}[scale=\largeGraphScale]
    \draw (1, 0) -- (0, 1);
    \draw (1.5, 0.5) -- (1, 0);
    \draw (1.5, 0.5) -- (1, 1);
    \draw (0,1) -- (1,1);
    \draw (0,0) -- (1,0);
    \draw (0,0) -- (0,1);
    \draw[fill=black] (0, 0) circle (0.125);
    \draw[fill=black] (0, 1) circle (0.125);
    \draw[fill=black] (1, 0) circle (0.125);
    \draw[fill=black] (1, 1) circle (0.125);
    \draw[fill=\specialVertexColor] (1.5, 0.5) circle (0.125);
\end{tikzpicture}
}

\newcommand{\trianglestick}{
\begin{tikzpicture}[scale=\largeGraphScale]
    \draw (1, 0) -- (1, 1);
    \draw (1.5, 0.5) -- (1, 0);
    \draw (1.5, 0.5) -- (1, 1);
    \draw (0,1) -- (0,0);
    \draw[fill=black] (0, 0) circle (0.125);
    \draw[fill=black] (0, 1) circle (0.125);
    \draw[fill=black] (1, 0) circle (0.125);
    \draw[fill=black] (1, 1) circle (0.125);
    \draw[fill=\specialVertexColor] (1.5, 0.5) circle (0.125);
\end{tikzpicture}
}

\newcommand{\bull}{
\begin{tikzpicture}[scale=\largeGraphScale]
    \draw (1, 0) -- (1, 1);
    \draw (1.5, 0.5) -- (1, 0);
    \draw (1.5, 0.5) -- (1, 1);
    \draw (1,0) -- (0,0);
    \draw (0,1) -- (1,1);
    \draw[fill=black] (0, 0) circle (0.125);
    \draw[fill=black] (0, 1) circle (0.125);
    \draw[fill=black] (1, 0) circle (0.125);
    \draw[fill=black] (1, 1) circle (0.125);
    \draw[fill=\specialVertexColor] (1.5, 0.5) circle (0.125);
\end{tikzpicture}
}

\newcommand{\bowtiebruh}{
\begin{tikzpicture}[scale=\largeGraphScale]
    \draw (1, 0) -- (1, 1);
    \draw (1.5, 0.5) -- (1, 0);
    \draw (1.5, 0.5) -- (1, 1);
    \draw (1,1) -- (0,0);
    \draw (0,1) -- (1,1);
    \draw (0,0) -- (0,1);
    \draw[fill=black] (0, 0) circle (0.125);
    \draw[fill=black] (0, 1) circle (0.125);
    \draw[fill=black] (1, 0) circle (0.125);
    \draw[fill=black] (1, 1) circle (0.125);
    \draw[fill=\specialVertexColor] (1.5, 0.5) circle (0.125);
\end{tikzpicture}
}

\newcommand{\cricket}{
\begin{tikzpicture}[scale=\largeGraphScale]
    \draw (1, 0) -- (1, 1);
    \draw (1.5, 0.5) -- (1, 0);
    \draw (1.5, 0.5) -- (1, 1);
    \draw (1,0) -- (0,0);
    \draw (0,1) -- (1,0);
    \draw[fill=black] (0, 0) circle (0.125);
    \draw[fill=black] (0, 1) circle (0.125);
    \draw[fill=black] (1, 0) circle (0.125);
    \draw[fill=black] (1, 1) circle (0.125);
    \draw[fill=\specialVertexColor] (1.5, 0.5) circle (0.125);
\end{tikzpicture}
}

\newcommand{\bullcaseone}{
\begin{tikzpicture}[scale=\largeGraphScale]
    \draw (0, 0) -- (0, 1);
    \draw (0, 1) -- (1, 0);
    \draw (1, 0) -- (1, 1);
    \draw (0, 1) -- (1, 1);
    \draw (1, 1) -- (1.5, 0.5);
    \draw[fill=black] (0, 0) circle (0.125);
    \draw[fill=black] (0, 1) circle (0.125);
    \draw[fill=black] (1, 0) circle (0.125);
    \draw[fill=black] (1, 1) circle (0.125);
    \draw[fill=\specialVertexColor] (1.5, 0.5) circle (0.125);
\end{tikzpicture}
}

\newcommand{\housenobrim}{
\begin{tikzpicture}[scale=\largeGraphScale]
    \draw (0, 0) -- (1, 1);
    \draw (0, 0) -- (1, 0);
    \draw (0, 0) -- (0, 1);
    \draw (0, 1) -- (1, 1);
    \draw (0, 1) -- (1, 0);
    \draw (1, 1) -- (1.5, 0.5);
    \draw (1, 0) -- (1.5, 0.5);
    \draw[fill=black] (0, 0) circle (0.125);
    \draw[fill=black] (0, 1) circle (0.125);
    \draw[fill=black] (1, 0) circle (0.125);
    \draw[fill=black] (1, 1) circle (0.125);
    \draw[fill=\specialVertexColor] (1.5, 0.5) circle (0.125);
\end{tikzpicture}
}

\newcommand{\nobutt}{
\begin{tikzpicture}[scale=\largeGraphScale]
    \draw (0, 0) -- (1, 1);
    \draw (0, 0) -- (1, 0);
    \draw (0, 1) -- (1, 1);
    \draw (0, 1) -- (1, 0);
    \draw (1,1) -- (1,0);
    \draw (1, 1) -- (1.5, 0.5);
    \draw (1, 0) -- (1.5, 0.5);
    \draw[fill=black] (0, 0) circle (0.125);
    \draw[fill=black] (0, 1) circle (0.125);
    \draw[fill=black] (1, 0) circle (0.125);
    \draw[fill=black] (1, 1) circle (0.125);
    \draw[fill=\specialVertexColor] (1.5, 0.5) circle (0.125);
\end{tikzpicture}
}

\newcommand{\fan}{
\begin{tikzpicture}[scale=\largeGraphScale]
    \draw (0, 0) -- (1, 0);
    \draw (0, 0) -- (0, 1);
    \draw (0, 1) -- (1, 1);
    \draw (0, 1) -- (1, 0);
    \draw (1,0) -- (1,1);
    \draw (1, 1) -- (1.5, 0.5);
    \draw (1, 0) -- (1.5, 0.5);
    \draw[fill=black] (0, 0) circle (0.125);
    \draw[fill=black] (0, 1) circle (0.125);
    \draw[fill=black] (1, 0) circle (0.125);
    \draw[fill=black] (1, 1) circle (0.125);
    \draw[fill=\specialVertexColor] (1.5, 0.5) circle (0.125);
\end{tikzpicture}
}

\newcommand{\hamsandwichwitharm}{
\begin{tikzpicture}[scale=\largeGraphScale]
    \draw (0, 0) -- (1, 0);
    \draw (0, 0) -- (0, 1);
    \draw (0, 1) -- (1, 1);
    \draw (1,0) -- (1,1);
    \draw (0, 1) -- (1, 0);
    \draw (1, 1) -- (1.5, 0.5);
    \draw[fill=black] (0, 0) circle (0.125);
    \draw[fill=black] (0, 1) circle (0.125);
    \draw[fill=black] (1, 0) circle (0.125);
    \draw[fill=black] (1, 1) circle (0.125);
    \draw[fill=\specialVertexColor] (1.5, 0.5) circle (0.125);
\end{tikzpicture}
}

\newcommand{\hamSandwichPlusIsolated}{
\begin{tikzpicture}[scale=\largeGraphScale]
    \draw (0, 0) -- (1, 0);
    \draw (0, 0) -- (0, 1);
    \draw (0, 1) -- (1, 1);
    \draw (1,0) -- (1,1);
    \draw (0, 1) -- (1, 0);
    \draw[fill=black] (0, 0) circle (0.125);
    \draw[fill=black] (0, 1) circle (0.125);
    \draw[fill=black] (1, 0) circle (0.125);
    \draw[fill=black] (1, 1) circle (0.125);
    \draw[fill=\specialVertexColor] (1.5, 0.5) circle (0.125);
\end{tikzpicture}
}
\newcommand{\hamSandwichWithOtherArm}{
\begin{tikzpicture}[scale=\largeGraphScale]
    \draw (0, 0) -- (1, 0);
    \draw (0, 0) -- (0, 1);
    \draw (0, 1) -- (1, 1);
    \draw (1,0) -- (1,1);
    \draw (0, 1) -- (1, 0);
    \draw (1, 0) -- (1.5, 0.5);
    \draw[fill=black] (0, 0) circle (0.125);
    \draw[fill=black] (0, 1) circle (0.125);
    \draw[fill=black] (1, 0) circle (0.125);
    \draw[fill=black] (1, 1) circle (0.125);
    \draw[fill=\specialVertexColor] (1.5, 0.5) circle (0.125);
\end{tikzpicture}
}
\newcommand{\smallHamSandwichPlusIsolated}{
\begin{tikzpicture}[scale=\graphScale]
    \draw (0, 0) -- (1, 0);
    \draw (0, 0) -- (0, 1);
    \draw (0, 1) -- (1, 1);
    \draw (1,0) -- (1,1);
    \draw (0, 1) -- (1, 0);
    \draw[fill=black] (0, 0) circle (0.125);
    \draw[fill=black] (0, 1) circle (0.125);
    \draw[fill=black] (1, 0) circle (0.125);
    \draw[fill=black] (1, 1) circle (0.125);
    \draw[fill=\specialVertexColor] (1.5, 0.5) circle (0.125);
\end{tikzpicture}
}

\newcommand{\emptyhouse}{
\begin{tikzpicture}[scale=\largeGraphScale]
    \draw (0, 0) -- (1, 0);
    \draw (0, 0) -- (0, 1);
    \draw (0, 1) -- (1, 1);
    \draw (1,0) -- (1,1);
    \draw (1, 1) -- (1.5, 0.5);
    \draw (1, 0) -- (1.5, 0.5);
    \draw[fill=black] (0, 0) circle (0.125);
    \draw[fill=black] (0, 1) circle (0.125);
    \draw[fill=black] (1, 0) circle (0.125);
    \draw[fill=black] (1, 1) circle (0.125);
    \draw[fill=\specialVertexColor] (1.5, 0.5) circle (0.125);
\end{tikzpicture}
}

\newcommand{\completeGraphWithArm}{
\begin{tikzpicture}[scale=\largeGraphScale]
    \draw (0, 0) -- (1, 0);
    \draw (0, 0) -- (1, 1);
    \draw (0, 0) -- (0, 1);
    \draw (1, 0) -- (1, 1);
    \draw (1, 0) -- (0, 1);
    \draw (0, 1) -- (1, 1);
    \draw (1.5, 0.5) -- (1, 1);
    \draw[fill=black] (0, 0) circle (0.125);
    \draw[fill=black] (0, 1) circle (0.125);
    \draw[fill=black] (1, 0) circle (0.125);
    \draw[fill=black] (1, 1) circle (0.125);
    \draw[fill=\specialVertexColor] (1.5, 0.5) circle (0.125);
\end{tikzpicture}
}
\newcommand{\completeGraphIsolatedVertex}{
\begin{tikzpicture}[scale=\largeGraphScale]
    \draw (0, 0) -- (1, 0);
    \draw (0, 0) -- (1, 1);
    \draw (0, 0) -- (0, 1);
    \draw (1, 0) -- (1, 1);
    \draw (1, 0) -- (0, 1);
    \draw (0, 1) -- (1, 1);
    \draw[fill=black] (0, 0) circle (0.125);
    \draw[fill=black] (0, 1) circle (0.125);
    \draw[fill=black] (1, 0) circle (0.125);
    \draw[fill=black] (1, 1) circle (0.125);
    \draw[fill=\specialVertexColor] (1.5, 0.5) circle (0.125);
\end{tikzpicture}
}
\newcommand{\smallCompleteGraphIsolatedVertex}{
\begin{tikzpicture}[scale=\graphScale]
    \draw (0, 0) -- (1, 0);
    \draw (0, 0) -- (1, 1);
    \draw (0, 0) -- (0, 1);
    \draw (1, 0) -- (1, 1);
    \draw (1, 0) -- (0, 1);
    \draw (0, 1) -- (1, 1);
    \draw[fill=black] (0, 0) circle (0.125);
    \draw[fill=black] (0, 1) circle (0.125);
    \draw[fill=black] (1, 0) circle (0.125);
    \draw[fill=black] (1, 1) circle (0.125);
    \draw[fill=\specialVertexColor] (1.5, 0.5) circle (0.125);
\end{tikzpicture}
}
\newcommand{\triangleIsolatedEdge}{
\begin{tikzpicture}[scale=\largeGraphScale]
    \draw (0, 0) -- (1, 0);
    \draw (0, 0) -- (0, 1);
    \draw (1, 0) -- (0, 1);
    \draw (1.5, 0.5) -- (1, 1);
    \draw[fill=black] (0, 0) circle (0.125);
    \draw[fill=black] (0, 1) circle (0.125);
    \draw[fill=black] (1, 0) circle (0.125);
    \draw[fill=black] (1, 1) circle (0.125);
    \draw[fill=\specialVertexColor] (1.5, 0.5) circle (0.125);
\end{tikzpicture}
}
\newcommand{\triangleWithTail}{
\begin{tikzpicture}[scale=\largeGraphScale]
    \draw (0, 0) -- (0, 1);
    \draw (0, 1) -- (1, 1);
    \draw (1,0) -- (1,1);
    \draw (1, 1) -- (1.5, 0.5);
    \draw (1, 0) -- (1.5, 0.5);
    \draw[fill=black] (0, 0) circle (0.125);
    \draw[fill=black] (0, 1) circle (0.125);
    \draw[fill=black] (1, 0) circle (0.125);
    \draw[fill=black] (1, 1) circle (0.125);
    \draw[fill=\specialVertexColor] (1.5, 0.5) circle (0.125);
\end{tikzpicture}
}
\newcommand{\triangleWithShortTail}{
\begin{tikzpicture}[scale=\largeGraphScale]
    \draw (0, 1) -- (1, 1);
    \draw (1,0) -- (1,1);
    \draw (1, 1) -- (1.5, 0.5);
    \draw (1, 0) -- (1.5, 0.5);
    \draw[fill=black] (0, 0) circle (0.125);
    \draw[fill=black] (0, 1) circle (0.125);
    \draw[fill=black] (1, 0) circle (0.125);
    \draw[fill=black] (1, 1) circle (0.125);
    \draw[fill=\specialVertexColor] (1.5, 0.5) circle (0.125);
\end{tikzpicture}
}

\newcommand{\zigZag}{
\begin{tikzpicture}[scale=\largeGraphScale]
    \draw (1, 0) -- (1, 1);
    \draw (1.5, 0.5) -- (1, 0);
    \draw (1,1) -- (0,0);
    \draw (0,1) -- (1,1);
    \draw (0,0) -- (0,1);
    \draw[fill=black] (0, 0) circle (0.125);
    \draw[fill=black] (0, 1) circle (0.125);
    \draw[fill=black] (1, 0) circle (0.125);
    \draw[fill=black] (1, 1) circle (0.125);
    \draw[fill=\specialVertexColor] (1.5, 0.5) circle (0.125);
\end{tikzpicture}
}

\newcommand{\triangleWithTwoTails}{
\begin{tikzpicture}[scale=\largeGraphScale]
    \draw (1, 0) -- (1, 1);
    \draw (1.5, 0.5) -- (1, 1);
    \draw (1,1) -- (0,0);
    \draw (0,1) -- (1,1);
    \draw (0,0) -- (0,1);
    \draw[fill=black] (0, 0) circle (0.125);
    \draw[fill=black] (0, 1) circle (0.125);
    \draw[fill=black] (1, 0) circle (0.125);
    \draw[fill=black] (1, 1) circle (0.125);
    \draw[fill=\specialVertexColor] (1.5, 0.5) circle (0.125);
\end{tikzpicture}
}

\newcommand{\badDart}{
\begin{tikzpicture}[scale=\largeGraphScale]
    \draw (0, 0) -- (0, 1);
    \draw (1, 0) -- (1, 1);
    \draw (1, 0) -- (0, 1);
    \draw (0, 1) -- (1, 1);
    \draw (1.5, 0.5) -- (1, 0);
    \draw (1.5, 0.5) -- (1, 1);
    \draw[fill=black] (0, 0) circle (0.125);
    \draw[fill=black] (0, 1) circle (0.125);
    \draw[fill=black] (1, 0) circle (0.125);
    \draw[fill=black] (1, 1) circle (0.125);
    \draw[fill=\specialVertexColor] (1.5, 0.5) circle (0.125);
\end{tikzpicture}
}

\newcommand{\fletching}{
\begin{tikzpicture}[scale=\graphScale]
    \draw (0, 0) -- (1, 0);
    \draw (1, 0) -- (1, 1);
    \draw (1, 0) -- (0, 1);
    \draw[fill=black] (0, 0) circle (0.125);
    \draw[fill=black] (0, 1) circle (0.125);
    \draw[fill=black] (1, 0) circle (0.125);
    \draw[fill=black] (1, 1) circle (0.125);
\end{tikzpicture}
}
\newcommand{\brokenFrame}{
\begin{tikzpicture}[scale=\graphScale]
    \draw (0, 0) -- (1, 0);
    \draw (1, 0) -- (1, 1);
    \draw[fill=black] (0, 0) circle (0.125);
    \draw[fill=black] (0, 1) circle (0.125);
    \draw[fill=black] (1, 0) circle (0.125);
    \draw[fill=black] (1, 1) circle (0.125);
\end{tikzpicture}
}

\makeatletter
\let\@fnsymbol\@arabic
\makeatother
\DeclareMathOperator{\Aut}{Aut}
\DeclareMathOperator{\pg}{\Gamma}
\DeclareMathOperator{\pgc}{\overline{\Gamma}}
\usepackage{accents}

\DeclareMathOperator{\GL}{GL}
\DeclareMathOperator{\Sz}{Sz}
\DeclareMathOperator{\PSL}{PSL}

\title{Classification of the Prime Graphs of $\Sz(8)$-, $\Sz(32)$-, and $\PSL(2, 2^5)$-Solvable Groups}

\author{Thomas Michael Keller\thanks{Thomas Michael Keller: keller@txstate.edu; Department of Mathematics, Texas State University, 601 University Drive, San Marcos, TX, 78666, USA.}, Zachary Martin\footnote{Zachary Martin: zvmartin@willamette.edu; Department of Mathematics, Willamette University, 3406 North 26th Street, Tacoma, WA, 98407, USA.}, Alexa Renner\footnote{Corresponding Author: Alexa Renner: renneram@rose-hulman.edu; Department of Mathematics, Rose-Hulman Institute of Technology, 5500 Wabash Ave., Terre Haute, IN 47803, USA.}, Gabriel Roca\footnote{Gabriel Roca: gabriel.roca@ucf.edu; Department of Mathematics, University of Central Florida, 4000 Central Florida Blvd, Orlando, FL 32816, USA.}, Eric Yu\footnote{Eric Yu: ericyu25@sas.upenn.edu; Department of Mathematics, University of Pennsylvania, 209 S 33rd St, Philadelphia, PA 19104, USA}}

\newtheorem{counter}{counter}[section]

\theoremstyle{definition}
\newtheorem{definition}[counter]{Definition}
\newtheorem*{definition*}{Definition}
\newtheorem{fact}[counter]{Fact}
\newtheorem{conjecture}[counter]{Conjecture}

\theoremstyle{plain}
\newtheorem{theorem}[counter]{Theorem}
\newtheorem{lemma}[counter]{Lemma}
\newtheorem{corollary}[counter]{Corollary}
\newtheorem{proposition}[counter]{Proposition}

\theoremstyle{remark}
\newtheorem{remark}[counter]{Remark}

\begin{document}
\maketitle
\begin{abstract}
    For a finite group $G$, the vertices of the prime graph $\Gamma(G)$ are the primes that divide $|G|$, and two vertices $p$ and $q$ are connected by an edge if there is an element of order $pq$ in $G$. Prime graphs of solvable groups have been classified, and prime graphs of groups whose noncyclic composition factors are isomorphic to a single nonabelian simple group $T$ have been classified in the case where $T$ has order divisible by exactly three or four distinct primes, except for the cases $T = \Sz(8)$, $T = \Sz(32)$, and $T = \operatorname{PSL}(2,q)$,
    which in some sense are the hardest cases.
    In this paper, we complete the classification for $T = \Sz(32)$, $T = \Sz(8)$, and $T = \PSL(2,2^5)$, with the latter two being the first cases ever studied where $|\text{Out}(T)|$ has prime factors which do not divide $|T|$. The groups studied in this paper are also the first ones requiring knowledge of their Brauer character tables to complete the classification task.
\end{abstract}

\hfill\break

\hfill\break
\section*{Introduction}

The object of interest in this paper is the prime graph of a finite group, also known as the Gruenberg-Kegel graph of the group. This graph has the primes dividing the group order as its vertices, and two vertices $p$ and $q$ are linked by an edge if and only if there exists an element of order $pq$ in $G$. Prime graphs of finite groups have been studied quite intensely since the 1970s, and this paper is a contribution to completely classifying prime graphs of certain families of finite groups.\\

Given a finite group $G$, we will write $\pg(G)$ to denote its prime graph, $\pgc(G)$ to denote the complement of its prime graph, and $\pi(G)$ to denote the prime divisors of $|G|$. We say that $G$ is $K_n$ if $G$ is simple and $|\pi(G)| = n$.

In 2015, the authors of \cite{2015REU} gave a complete classification of the prime graphs of solvable groups:
\begin{theorem} \thlabel{solvclass}
    \cite[Theorem 2.10]{2015REU} An unlabeled simple graph $\Xi$ is isomorphic to the prime graph complement of a solvable group if and only if it is 3-colorable and triangle-free.
\end{theorem}

Since then, further research has been done on what happens when a single nonabelian simple group is allowed in as a factor in the composition series.
\begin{definition}
    Let $T$ be a nonabelian simple group, and let $G$ be a group. We say that $G$ is $T$-solvable if all of its composition factors are either cyclic or isomorphic to $T$. We say that $G$ is strictly $T$-solvable if it is $T$-solvable and at least one of its composition factors is isomorphic to $T$.
\end{definition}
Efforts to classify the prime graphs of $T$-solvable groups began in the 2024 paper \cite{Huang}, where the case of $T = A_5$ was completed. In 2023, the authors of \cite{2022REU} classified the prime graphs of $T$-solvable groups for all remaining $K_3$ groups $T$. In 2024, the authors of \cite{2023REU} did the same for all $K_4$ groups, with the exception of $\Sz(8)$, $\Sz(32)$, and $\PSL(2,q)$ where $q$ is a prime power such that $|\PSL(2,q)| = \frac{q(q^2-1)}{(2,(q-1))}$ has exactly four prime divisors. 
$\Sz(32)$ is a challenge simply due to its size. It has order 32537600, which is too large for many computational techniques. The difficulty with $\Sz(8)$ and $\PSL(2,q)$ is more theoretical. Many of the algebraic techniques used in previous papers relied on the assumption that $\pi(T) = \pi(\Aut(T))$, which does not hold for these particular groups:
\begin{remark} \cite[Lemma 1.5]{2023REU}
    The only $K_4$-groups $T$ such that $\pi(T)\neq\pi(\Aut(T))$ are the Suzuki group $\Sz(8)$ and certain groups $\PSL(2, p^f)$ where $p$ is prime.
\end{remark}

In this paper we classify the prime graphs of $\Sz(8)$-, $\Sz(32)$-, and $\PSL(2, 2^5)$-solvable groups using new methods from modular character theory. Namely, we will use fixed point information extracted from the Brauer character tables of these groups. We use \cite[Lemma 6.2]{BrauerTableFP}, which we restate below for the reader's convenience. For a proof, see \cite[Lemma 4]{C55Groups}.
\begin{lemma}\thlabel{MaslovaPaperLemma}
    Let $G$ be a finite simple group, $F$ be a field of characteristic $p > 0$, $V$ be an absolutely irreducible $GF$-module, and $\beta$ be a Brauer character of $V$. If $g\in G$ is an element of prime order distinct distinct from $p$, then:
    \begin{align*}
        \operatorname{dim} C_V(g) = (\beta_{\langle g\rangle}, 1_{\langle g\rangle}) = \frac{1}{|g|}\sum_{x\in\langle g\rangle} \beta(x).
    \end{align*}
\end{lemma}
We remark that more work on the case
of the groups $\PSL(2, q)$ will be forthcoming, but we have decided to 
include the smallest such $K_4$-group, $\PSL(2, 32)$, in this paper because
the proof of the corresponding classification result is very similar to
the case of $\Sz(8)$.\\

The main classification results of this paper are Theorems \ref{Sz32Classification}, \ref{Sz8Classification}, and \ref{PSL232Classification} for $\Sz(32)$, $\Sz(8)$, and $\PSL(2, 2^5)$, respectively. Because they are technical, we do not state them here. To give the reader an idea of the flavor of these classifications, we include below the prime graph complement of an $\Sz(8)$-solvable group on 10 vertices with a maximal number of edges:
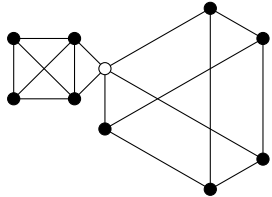
\begin{center}
    \maxEdgeDiagram
\end{center}
The subgraph induced by the white vertex and all vertices to the left of it is $\pgc(\Sz(8))$, and the subgraph induced by the white vertex and all vertices to the right of it is a triangle-free and 3-colorable graph on 6 vertices.

The paper is organized as follows: Section \ref{sec:PGAE} proves some foundational results on prime graphs of abelian extensions. Section \ref{sec:Suzuki32} classifies the prime graphs of $\Sz(32)$-solvable groups. 
In Section \ref{sec:RIR} we prove a result which allows us to extract the prime graph complements of extensions of a nonabelian simple group $T$ by a $p$-group for $p\in\pi(T)$ from irreducible Brauer characters of $T$. Section \ref{sec:subgraphsInduced} proves results that allow us to restrict our attention to specific subgroups of a $T$-solvable group when considering certain subgraphs of $\pgc(G)$. Section \ref{sec:UsefulLemmas} provides several results that we use frequently in the classifications of prime graph complements of $\Sz(8)$- and $\PSL(2, 2^5)$-solvable groups. In Sections \ref{Sz8} and \ref{sec:PSL} we finally classify the prime graph complements of $\Sz(8)$- and $\PSL(2, 2^5)$-solvable groups, respectively. Some avenues for future work are presented in Section \ref{sec:Outlook}. 

Appendix \ref{AppendixA} lists relevant graphs on 5 vertices, Appendix \ref{AppendixB} elaborates on some of the computations needed to produce the previous sections, and Appendix \ref{AppendixC} provides a link to a repository of computer programs that we used. \\

We will now introduce some notation and conventions that will be used throughout this paper. 
\begin{itemize}
    \item When we refer to a Brauer table mod $p$ of a finite group $G$, we mean a table of irreducible Brauer characters mod $p$ of $G$ evaluated at conjugacy classes of $G$ with elements of order not divisible by $p$.
    \item For a graph $\Xi$, we will denote the vertex set of $\Xi$ by $V(\Xi)$ and edge set of $\Xi$ by $E(\Xi)$. We will often write $p-q \in \Xi$ as shorthand for $p-q \in E(\Xi)$.
    \item Let $G$ be a group, and let $\pi_0$ be a set of primes. We say that $G$ is a $\pi_0$-group if $\pi(G) \subseteq \pi_0$. We say that $G$ is a strict $\pi_0$-group if $\pi(G) = \pi_0$. 
    \item For a graph $\Xi$ and $X\subseteq V(\Xi)$, we will write $\Xi[X]$ to denote the subgraph of $\Xi$ induced by $X$. 
    \item If $X$ is a set, $Y$ is a group, and if  $x\in X, y \in Y$ and there is a function $\varphi: X \to \Aut(Y)$, then we will write $\varphi_x$ in place of $\varphi(x)$ so that $\varphi_x(y) \in Y$. 
    \item We will often write normal series in ATLAS notation: $G = X_1.X_2. \cdots .X_k$ if there exists a normal series $G = N_k \trianglerighteq N_{k-1} \trianglerighteq \cdots \trianglerighteq N_1 \trianglerighteq N_0 = \{1\}$ such that $X_i \cong N_i/N_{i-1}$ for $i \in \{1, \cdots, k\}$. 
    \item Given labeled graphs $\Xi$ and $\Lambda$, we define their {\it labeled graph product} $\Xi \sqcap \Lambda$ as follows: Let $K$ be the complete graph on $V(\Xi)\cup V(\Lambda)$ (vertices with the same label are considered identical). Then $V(\Xi \sqcap \Lambda) = V(K)$, and $E(\Xi \sqcap \Lambda) = E(K) \setminus (E(\overline{\Xi})\cup E(\overline{\Lambda}))$. This operation is associative and commutative.

    Note that in the case where $V(\Xi) = V(\Lambda)$, we have $E(\Xi \sqcap \Lambda) = E(\Xi) \cap E(\Lambda)$.
\end{itemize}
We state an important definition:

\begin{definition} \thlabel{frobcrit}
    \cite[Definition 2.2.1]{2023REU} Let $P$ be a $p$-group. $P$ is said to satisfy the Frobenius Criterion if $P$ is cyclic, dihedral, Klein-4, or generalized quaternion. Note that every quotient of a generalized quaternion or cyclic group satisfies the Frobenius Criterion.
\end{definition}

Finally, we point out a change in our style from the previous papers in this series. We now place the focus on classifying the prime graph complements of $T$-solvable groups, rather than the prime graphs. Thus our statements will be written as: ``$\Xi$ is the prime graph complement of a $T$ solvable group if and only if..."

\section{Prime Graphs of Abelian Extensions}\label{sec:PGAE}
For convenience, we state the following elementary lemmas:
\begin{lemma}\thlabel{orderNotDividing}
    Let $G$ be a finite group, let $N \trianglelefteq G$ be a normal subgroup, and let $g\in G$ be an element. Then $o(g)$ divides $o(gN)|N|$.
\end{lemma}
\begin{proof}
    We see that $g^{o(gN)} \in N$, meaning $(g^{o(gN)})^{|N|} = g^{o(gN)|N|} = 1$. Thus, $o(g)$ divides $o(gN)|N|$.
\end{proof}

\begin{lemma}\thlabel{EdgeAdd}\thlabel{divisors}
    Let $G$, $N$, and $T$ be finite groups such that $G \cong N.T$, and let $r,s$ be primes dividing $|T|$. If $r-s \in \pgc(T)$ and neither $r$ nor $s$ divide $|N|$, then $r-s \in \pgc(G)$.
\end{lemma}
\begin{proof}
    Suppose that $r-s \notin \pgc(G)$. Then let $g \in G$ be an element of order $rs$. By \thref{orderNotDividing} we have that $rs$ divides $o(gN)|N|$. Since neither $r$ nor $s$ divides $|N|$, we have that $rs$ divides $o(gN)$, a contradiction since $G/N \cong T$ has no elements of order $rs$.
\end{proof}

\begin{lemma}\thlabel{reverseDivisors}
    Let $G$, $N$, and $T$ be finite groups such that $G \cong N.T$, and let $r,s$ be primes dividing $|N|$. If $r-s \in \pgc(N)$ and neither $r$ nor $s$ divide $|T|$, then $r-s \in \pgc(G)$.
\end{lemma}
\begin{proof}
    Suppose that $r-s \notin \pgc(G)$. Then let $g\in G$ be an element of order $rs$. We must have that $g \in N$ or else $o(g)$ would share divisors with $|T|$. So $g$ is an element of order $rs$ in $N$, a contradiction.
\end{proof}

\begin{theorem}\thlabel{primeGraphOfAbelianExtension}
    Let $G$ be a finite group, and let $A$ be an abelian normal subgroup of $G$. Consider the action $\varphi: G/A \to \Aut(A)$ defined by $\varphi_{gA}(a) = g^{-1}ag$. Then the prime graph $\pg(G)$ can be characterized as follows:
    \begin{enumerate}
        \item If $p,r \in \pi(A)$, then $p-r \in \pg(G)$. 
        \item If $p\in \pi(G)$ and $r \in \pi(G) \setminus \pi(A)$, then $p-r \in \pg(G)$ if and only if one of the following holds:
        \begin{enumerate}
            \item $p-r\in \pg(G/A)$.
            \item There exists $gA \in G/A$ of order $r$ such that $\varphi_{gA}$ fixes an element $a \in A$ of order $p$.
        \end{enumerate}
    \end{enumerate}
\end{theorem}
\begin{proof}
    Since $A$ is abelian, (1) is obvious. For the forward direction of (2), given that $p-r \in \pg(G)$, we will assume $p-r \notin \pg(G/A)$ and show that (2b) holds. Let $g\in G$ be an element of order $pr$. By \thref{orderNotDividing}, $o(gA)$ must be a multiple of $r$. Furthermore, $o(gA) \neq pr$ since $p-r \notin \pg(G/A)$. The only remaining possibility is that $o(gA) = r$, which tells us $g^r \in A$. Letting $a = g^r$, we see that $a$ is fixed under $\varphi_{gA}$, and $a \neq 1$ since $g$ has order $pr$. Thus, (2b) is satisfied.

    For the backward direction of (2), $p-r \in \pg(G/A) \implies p-r \in \pg(G)$ is clear. It remains to show that (2b) implies $p-r \in \pg(G)$. Let $gA \in G/A$ and $a \in A$ be elements satisfying (2b). Then $g$ has order a multiple of $r$. Because $a = \varphi_{gA}(a) = g^{-1}ag$, we know that $a$ commutes with $g$. So $\langle a, g \rangle$ is an abelian subgroup of $G$ with order a multiple of $pr$. Thus it has an element of order $pr$, and $p-r\in \pg(G)$.
\end{proof}
It will be useful to restate \thref{primeGraphOfAbelianExtension} in terms of the prime graph complement:
\begin{corollary} \thlabel{representationSemidirectProduct}
    Let $G$ be a finite group, and let $A$ be an abelian normal subgroup of $G$. Consider the action $\varphi: G/A \to \Aut(A)$ defined by $\varphi_{gA}(a) = g^{-1}ag$. Then the prime graph complement $\pgc(G)$ can be characterized as follows:
    \begin{enumerate}
        \item If $p,r \in \pi(A)$, then $p-r \notin \pgc(G)$. 
        \item If $p\in \pi(G)$ and $r \in \pi(G) \setminus \pi(A)$, then $p-r \notin \pgc(G)$ if and only if one of the following hold:
        \begin{enumerate}
            \item $p \in \pi(G/A)$ and $p-r \notin \pgc(G/A)$.
            \item There exists $gA \in G/A$ of order $r$ such that $\varphi_{gA}$ fixes an element $a \in A$ of order $p$.
        \end{enumerate}
    \end{enumerate}
\end{corollary}

\begin{corollary} \thlabel{split}
    Since we are able to recover $\pgc(G)$ using just $A$, $G/A$, and $\varphi$, we conclude that any two extensions of $G/A$ by $A$ with the same action have the same prime graph. In particular, $\pgc(G) = \pgc(\frac{G}{A} \ltimes_\varphi A)$.
\end{corollary}

\begin{definition} \thlabel{semidirectProductWithVectorSpace}
    Let $G$ be a group, let $\mathbb{F}$ be a field, and let $\varphi: G \to \GL(n, \F)$ be a representation. Let $V = \F^n$ be the vector space upon which $\varphi(G)$ acts naturally. We will write $V^+$ to denote the additive group of $V$. Since $\GL(n, \F) = \GL(V) \subseteq \Aut(V^+)$, we have that $\varphi$ induces an action of $G$ on $V^+$. We define $G \ltimes_\varphi V$ to equal $G \ltimes V^+$ under this action.
\end{definition}
\begin{remark}
    Since $G\ltimes_\varphi V$ is an abelian extension of $G$ by $V^+$ with conjugation action $\varphi$, \thref{representationSemidirectProduct} applies.
\end{remark}
\begin{proposition} \thlabel{elementaryAbelianExtensionsToRepresentations}
    Let $T$ be a finite group, let $q$ be a prime number, and let $A\cong (C_q)^n$ be an elementary abelian $q$-group. Then for every extension $G \cong A.T$, there is a representation $\varphi: T \to \GL(n,q)$ such that $\pgc(G) = \pgc(T \ltimes_\varphi (\F_q)^n)$.
\end{proposition}
\begin{proof}
    By \thref{split}, we can assume that $G \cong T \ltimes A$. Identify $A$ with the additive group of $(\F_q)^n$. Let $\varphi: T \to \Aut(A)$ be the action in $G$. It remains to show that $\varphi$ is a representation. This is true because $\Aut(A) = \Aut(((\F_q)^n)^+) = \GL(n,q)$.
\end{proof}

\section{The Suzuki Group Sz(32)}\label{sec:Suzuki32}
\begin{figure}[H]
    \centering
    \begin{minipage}{.5\textwidth}
        \centering
        \begin{tikzpicture}
            \coordinate (A) at (0, 1);   
            \coordinate (B) at (1, 1);   
            \coordinate (C) at (1, 0);   
            \coordinate (D) at (0, 0);   
            \draw (B) -- (C);
            \draw (B) -- (D);
            \draw (C) -- (D);
            \draw (A) -- (D);
            \draw (A) -- (B);
            \draw (A) -- (C);
            \foreach \point in {A, B, C, D}
                \draw[draw, fill=white] (\point) circle [radius=0.2cm];
            \foreach \point/\name in {(A)/2, (B)/31, (C)/5, (D)/41}
                \node at \point {\scriptsize\name};
        \end{tikzpicture}
        \caption*{$\pgc(\Sz(32))$}
    \end{minipage}%
    \begin{minipage}{.5\textwidth}
        \centering
        \begin{tikzpicture}
            \coordinate (A) at (0, 1);   
            \coordinate (B) at (1, 1);   
            \coordinate (C) at (1, 0);   
            \coordinate (D) at (0, 0);   
            \draw (B) -- (C);
            \draw (B) -- (D);
            \draw (C) -- (D);
            \draw (A) -- (D);
            \draw (A) -- (B);
            \foreach \point in {A, B, C, D}
                \draw[draw, fill=white] (\point) circle [radius=0.2cm];
            \foreach \point/\name in {(A)/2, (B)/31, (C)/5, (D)/41}
                \node at \point {\scriptsize\name};
        \end{tikzpicture}
        \caption*{$\pgc(\Aut(\Sz(32)))$}
    \end{minipage}
\end{figure}
In this section, we prove the results necessary for a complete classification of the prime graph complements of $\Sz(32)$-solvable groups and present their
classification in Theorem \ref{Sz32Classification}. It is a fact that $\pi(\Sz(32)) = \pi(\Aut(\Sz(32)))$ \cite{2023REU}. It is known that $\operatorname{Out(Sz(32))} \cong C_5$ and the Schur multiplier of Sz(32) is trivial \cite{2023REU}. The order of Sz(32) is large ($2^{10}\cdot5^2\cdot31\cdot41$), which made relying on GAP computations to realize certain prime graph complements of Sz(32)-solvable groups infeasible. To escape this difficulty, we turned to modular character theory. Additionally, using GAP to calculate the fixed points for \thref{totaldisconnect} was also infeasible due to the size of the group, so we used a bounding argument. 
\begin{theorem}\thlabel{totaldisconnect}
    Let $G$ be an $\Sz(32)$-solvable group. Then $\pgc(G)$ has no edges of the form $r - p$ for $r\in\pi(\Sz(32))$ and $p\in\pi(G) \setminus \pi(\Sz(32))$.
\end{theorem}
\begin{proof}
    We first consider the case where $r=2$. Sylow 2-subgroups of $\Sz(32)$ do not satisfy the Frobenius Criterion \thref{frobcrit}. Then by \cite[Proposition 2.2.2]{2023REU}, for all $p\in \pi(G)\setminus\pi(\Sz(32))$, it is clear that $2-p \notin \overline{\Gamma}(G)$. 
    
    So now assume $r \in \{5, 31, 41 \}$. The Schur multiplier of $\Sz(32)$ is 1, so $\Sz(32)$ has no nontrivial perfect central extensions. By \cite[Corollary 2.2.6]{2023REU}, it suffices to show that for every irreducible representation $\varphi: \Sz(32) \to \GL(V)$ where $V$ is a vector space over $\C$, there exists some element $x \in \Sz(32)$ of order $r$ such that $\varphi_x$ fixes some point in $V \setminus \{0\}$. Let $\chi$ be the character corresponding to $\varphi$. By the formula in the explanation below Table 3 in \cite{2023REU}, for all $x\in \Sz(32)$ of order $r$ we have $\dim(C_V(x)) = \frac{1}{r}\mleft(\sum_{j=1}^r\chi(x^j)\mright)$. We will show our argument for $r = 31$.

    Our goal is to show that for all elements $x \in \Sz(32)$ of order 31 and all irreducible characters $\chi$, we have $\frac{1}{31}\mleft(\sum_{j=1}^{31}\chi(x^j)\mright) > 0$. We consider the character table of $\Sz(32)$, which can be found in GAP \cite{GAP} with the command \verb|Display(CharacterTable("Sz(32)"));|. 
    \begin{itemize}
        \item For the trivial character $\chi = \chi_1$, we have $\chi(x) = 1$ for all $x \in \Sz(32)$, so $\frac{1}{31}\mleft(\sum_{j=1}^{31}\chi(x^j)\mright) = \frac{1}{31}\mleft(\sum_{j=1}^{31}1\mright)= 1 > 0$.
        \item For $\chi \in \{\chi_2, \cdots, \chi_{13}\}$, we have that $\chi(1) = 775$ and $\chi(x) = 0$ for all $x \in \Sz(32)$ of order 31. So $\frac{1}{31}\mleft(\sum_{j=1}^{31}\chi(x^j)\mright) = \frac{1}{31}(775) > 0$.
        \item For $\chi \in \{\chi_{14}, \cdots, \chi_{29}\}$, we have that $\chi(1) = 1024$. For all $x \in \Sz(32)$ of order 31, $\chi(x)$ is a complex number equal to the sum of two roots of unity, which we'll call $u_x$ and $v_x$. So 
        \begin{align}
            \textstyle
            \mleft| \frac{1}{31}\mleft(\sum_{j=1}^{31}\chi(x^j)\mright) \mright| &=  \textstyle \frac{1}{31} \mleft| 1024 + \sum_{j = 1}^{30}(u_{x^j} + v_{x^j}) \mright| \\ 
            &\geq \textstyle \frac{1}{31} \mleft( 1024 - \sum_{j = 1}^{30}(|u_{x^j}| + |v_{x^j}|) \mright) \\
            &= \textstyle \frac{1}{31} \mleft( 1024 - 60 \mright) \\
            &> 0.
        \end{align}

        \item For $\chi \in \{\chi_{30}, \cdots, \chi_{35}\}$, we have that $\chi(1) = 1271$ and $\chi(x) = 0$ for all $x \in \Sz(32)$ of order 31. So $\frac{1}{31}\mleft(\sum_{j=1}^{31}\chi(x^j)\mright) = \frac{1}{31}(1271) > 0$.
    \end{itemize}
    The cases $r=5$ and $r=41$ follow similarly.
\end{proof}

The next theorem relates the fixed points of Brauer characters to prime graph complements of $T$-solvable groups for $T$ a finite simple group. We use it primarily to construct $T$-solvable groups with specific prime graphs that would be too computationally expensive to test.
\begin{theorem} \thlabel{graphExistenceFromBrauerTable}
        Let $T$ be a finite simple group, let $p \in \pi(T)$, and let $\chi$ be an irreducible Brauer character of $T$ mod $p$. Define $R \subseteq \pi(T)$ to be the set of primes $r$ for which there exists some $t\in T$ of order $r$ such that 
        $\frac{1}{\operatorname{o}(t)}\sum_{x \in \langle t \rangle}\chi(x) > 0$.
        Then there exists a finite strictly $T$-solvable group $G$ such that $\overline{\Gamma}(G) = \overline{\Gamma}(T) \setminus\{p - r \mid r\in R\}$.
\end{theorem}
\begin{proof}
    Let $\varphi: T\to \GL(n, \mathbb{F})$ be an irreducible representation corresponding to $\chi$, where $\mathbb{F}$ is the algebraic closure of $\mathbb{F}_p$ by \cite[Lemma 2.1]{Navarro}. 
    By \cite[Corollary 9.4]{IsaacsCharacters}, $\mathbb{F}$ is a splitting field for every group. Thus $\varphi$ is absolutely irreducible, and the vector space $V = \mathbb{F}^n$ on which $\varphi(T)$ acts is an absolutely irreducible $\mathbb{F}T$-module. By \thref{MaslovaPaperLemma}, for all $t \in T$ of prime order coprime to $p$ we have that $\frac{1}{o(t)}\sum_{x \in \langle t \rangle}\chi(x) = \operatorname{dim}C_V(t)$. Thus, $R$ is the set of primes $r$ for which there exists some $t \in T$ of order $r$ such that $\varphi(t)$ fixes a point in $V\setminus \{0\}$. For each $r \in R$, fix a representative $t_r \in T$ with this property.

    By \cite[Page 22]{Brawley} we have that 
    $\mathbb{F}$ is the union of the tower of fields 
    \begin{equation}
        \F_{p^{1!}} \subseteq \F_{p^{2!}} \subseteq \F_{p^{3!}} \subseteq \cdots.
    \end{equation}

    Let $U\subseteq \F$ be the set of all entries of every matrix in $\varphi(T)$. Since $U$ is finite, there exists some $m \in \Z^+$ such that $U \subseteq \F_{2^{m!}}$. So $\varphi$ induces an action of $T$ on $\widetilde{V} = (\F_{2^{m!}})^n$. Recall that a linear transformation has a nonzero fixed point if and only if 1 is an eigenvalue. For all $r\in R$ we see that $\varphi(t_r)$ still fixes points in $\widetilde{V}\setminus \{0\}$ since the characteristic polynomial of the matrix corresponding to $\varphi(t_r)$ has not changed, meaning 1 is still an eigenvalue. Letting $G = T \ltimes_\varphi \widetilde{V}$, we see that $G$ is finite, and by \thref{representationSemidirectProduct}, $\pgc(G) = \pgc(T)\setminus\{p-r \mid r\in R\}$.
\end{proof}
\begin{corollary}\thlabel{ElusiveGraphExistence}
        There is an $\Sz(32)$-solvable group which realizes the graph \SuzukiThirtyTwoElusiveGraph on four vertices consisting of a triangle plus one edge.

\end{corollary}
\begin{proof}
    We consider the Brauer table of $\Sz(32)$ mod 2, which can be found in GAP \cite{GAP} with the command \verb|Display(BrauerTable("Sz(32)",2));|. Consider the Brauer character $\chi_{17}$. We observe that there exist elements $t \in \Sz(32)$ of orders 5 and 31 respectively such that $\frac{1}{o(t)}\sum_{x \in \langle t \rangle}\chi_{17}(x) > 0$, but there is no such $t$ of order 41. It follows from \thref{graphExistenceFromBrauerTable} that there exists a finite $\Sz(32)$-solvable group $G$ such that $\pgc(G) \cong \SuzukiThirtyTwoElusiveGraph$.
\end{proof}
\begin{lemma} \thlabel{4VertexGraphs}
    All graphs on 4 vertices can be realized as the prime graph complement of an $\Sz(32)$-solvable group.
\end{lemma}
\begin{proof}
    Let $\Xi$ be any graph on 4 vertices. If $\Xi$ is triangle-free and 3-colorable, then by \thref{solvclass} it is the prime graph complement of some solvable group. Up to isomorphism, there are four graphs on 4 vertices that are not both triangle-free and three-colorable:
    \begin{itemize}
        \item The complete graph \fourVertexCompleteGraph can be realized by $\Sz(32)$ itself.
        \item The complete graph minus one edge \hamSandwich can be realized by $\Aut(\Sz(32))$.
        \item The triangle plus one isolated vertex \trianglePlusIsolated can be realized by $\Sz(32) \times C_2$.
        \item The triangle plus one edge \SuzukiThirtyTwoElusiveGraph can be realized by \thref{ElusiveGraphExistence}.
    \end{itemize}
\end{proof}
We now are ready to state and prove the classification result for Sz(32)-solvable groups.
\begin{theorem}\thlabel{Sz32Classification}
Given a graph $\Xi$, we have that $\Xi$ is isomorphic to the prime graph complement of some $\Sz(32)$-solvable group if and only if one of the following is true: 
\begin{enumerate}
    \item $\Xi$ is triangle-free and 3-colorable.
    \item There exists a subset $X = \{a,b,c,d\} \subseteq V(\Xi)$ equal to its own closed neighborhood in $\Xi$ such that $X$ contains at least one triangle and $\Xi \setminus X$ is triangle-free and 3-colorable.
\end{enumerate}
\end{theorem}
\begin{proof}
Let $G$ be an $\Sz(32)$-solvable group with prime graph complement isomorphic to $\Xi$. If $G$ is solvable, then $\Xi$ is triangle-free and 3-colorable by \thref{solvclass}, so (1) holds. Otherwise, $G$ is strictly $\Sz(32)$-solvable, and we may use \thref{totaldisconnect} to find that there are no edges between $\pi(\Sz(32))$ and $V(\Xi) \setminus \pi(\Sz(32))$. Letting $X = \pi(\Sz(32))$, we see that $X$ equals its own closed neighborhood in $\Xi$. By \cite[Lemma 2.1.2]{2023REU}, we have that there is a solvable subgroup $N \leq G$ such that $\pi(N) \supseteq \pi(G) \setminus \pi(\Sz(32))$. Therefore, $\Xi\setminus X$ is triangle-free and 3-colorable.


Conversely, if $\Xi$ satisfies (1), then by \thref{solvclass} it can be realized by a solvable group. If $\Xi$ satisfies (2), then using \thref{4VertexGraphs} we can construct a group $G$ whose prime graph complement is isomorphic to the subgraph induced by $X$. Using the methods developed in \cite[Theorem 2.8]{2015REU}, we can construct a solvable group $N$ such that $(|N|, |G|)=1$ and $\overline{\Gamma}(N) \cong \Xi \setminus X$. Then $\pgc(G\times N) \cong \Xi$.
\end{proof}


\section{Reduction to Irreducible Representations}\label{sec:RIR}
The main goal of this section is to prove \thref{pGroupExtensionToGraphIntersection}, which states that given a finite group $T$ and a prime number $p\in \pi(T)$, the prime graph complements realizable by an extension of $T$ by a $p$-group are exactly those that can be written as a labeled graph product of prime graphs of groups arising from irreducible representations of $T$ over the field $\F_p$. Using \thref{pGroupExtensionToGraphIntersection} and \thref{graphExistenceFromBrauerTable} we prove \thref{GeneralizedModuleExtensions}, which allows us to extract these possible graphs from the Brauer table of $T$ mod $p$. 

\begin{lemma}\thlabel{Intersections}
        Let $N_1$, $N_2$, and $G$ be finite groups where $N_1$ and $N_2$ are abelian, and let $\varphi^1: G\to \Aut(N_1)$ and $\varphi^2: G \to \Aut(N_2)$ be homomorphisms. Let $\varphi^1 \times \varphi^2: G \to \Aut(N_1 \times N_2)$ be defined by $(\varphi^1\times \varphi^2)_g((n_1, n_2)) = (\varphi^1_g(n_1), \varphi^2_g(n_2))$. Then
        \begin{equation}
            \pgc(G \ltimes_{\varphi^1 \times \varphi^2} (N_1 \times N_2)) =
            \left(\pgc(G\ltimes_{\varphi^1}N_1) \sqcap \pgc(G\ltimes_{\varphi^2}N_2)\right) \setminus \{ r-p \mid r \in \pi(N_1), p \in \pi(N_2) \}.
        \end{equation}
\end{lemma}
\begin{proof}
    By \thref{representationSemidirectProduct}, it suffices to show given any $g\in G$ that $(\varphi^1 \times \varphi^2)_g$ is fixed point free if and only if both $\varphi^1_g$ and $\varphi^2_g$ are fixed point free. If $(\varphi^1 \times \varphi^2)_g$ fixes $(n_1, n_2) \in N_1 \times N_2 \setminus \{(1,1)\}$, then assuming without loss of generality that $n_1 \neq 1$, we see that $\varphi^1_g$ fixes $n_1$. Conversely, if $\varphi^1_g$ fixed some $n_1 \in N_1\setminus \{1\}$, then $(\varphi^1 \times \varphi^2)_g$ would fix $(n_1, 1)$.
\end{proof}

\begin{corollary} \thlabel{directSumIntersection}
    Let $G$ be a finite group, let $V_1$ and $V_2$ be vector spaces over a finite field, and let $\varphi^1: G \to \GL(V_1)$ and $\varphi^2: G\to \GL(V_2)$ be representations of $G$. Let $\varphi^1 \oplus \varphi^2: G \to \GL(V_1 \oplus V_2)$ be defined by $(\varphi^1 \oplus \varphi^2)_g((v_1,v_2)) = (\varphi^1_g(v_1), \varphi^2_g(v_2))$. Then 
    \begin{equation}
        \pgc(G \ltimes_{\varphi^1 \oplus \varphi^2} (V_1 \oplus V_2)) = \pgc(G \ltimes_{\varphi^1} V_1) \sqcap \pgc(G \ltimes_{\varphi^2} V_2).
    \end{equation}
\end{corollary}

\begin{lemma} \thlabel{splitRepresentation}
    Let $G$ be a finite group and $V$ be a vector space over a finite field $K$ of characteristic $q$. Given a reducible representation $\varphi: G \to \GL(V)$, there exists a $K$-vector space $\widetilde{V}$ of the same dimension as $V$ and a decomposable representation $\widetilde{\varphi}: G \to \GL(\widetilde{V})$ such that $\pgc(G \ltimes_\varphi V) = \pgc(G \ltimes_{\widetilde{\varphi}} \widetilde{V})$.
\end{lemma}
\begin{proof}
    Since $\varphi$ is reducible, $V$ has a $\varphi_G$-invariant nontrivial proper subspace $W$. We define $\widetilde{V} = W \oplus V/W$ and $\widetilde{\varphi}: G \to \GL(\widetilde{V})$ by $\widetilde{\varphi}_g((w, v+W)) = (\varphi_g(w), \varphi_g(v) + W)$ for $v\in V$ and $w\in W$. We must now show that $\pgc(G \ltimes_\varphi V) = \pgc(G \ltimes_{\widetilde{\varphi}} \widetilde{V})$. By \thref{representationSemidirectProduct}, it suffices to show given any $g\in G$ of prime order $p \neq q$ that $\varphi_g$ has fixed points on $V$ if and only if $\widetilde{\varphi}_g$ has fixed points on $\widetilde{V}$. 

    Forward direction: Given $\varphi_g(v) = v$ for some $v \in V\setminus \{0\}$, we split into two cases:
    \begin{itemize}
        \item If $v\in W$, then $\widetilde{\varphi}_g$ fixes $(v, 0+W)$. 
        \item If $v\notin W$, then $\widetilde{\varphi}_g$ fixes $(0, v+W)$.
    \end{itemize}
    
    Backward direction: Given $\widetilde{\varphi}_g((w, v+W)) = (w, v+W)$ for some $(w, v+W) \in \widetilde{V} \setminus \{0\}$, we split into two cases:
    \begin{itemize}
        \item If $w \neq 0$, then $\varphi_g$ fixes $w$.
        \item If $v+W\neq 0+W$, then $v\notin W$. We have that $\varphi_g(v) + W = v+W$, meaning $\varphi_g(v) = v+w_1$ for some $w_1 \in W$. Since o$(g) = p$ which is coprime to $q$, we can apply Maschke's theorem to get some $\varphi_g$-invariant subspace $U \leq V$ such that $V = W \oplus U$. Write $v = w_0 + u_0$ with $w_0 \in W$, $u_0 \in U$. Since $v \notin W$, we have $u_0 \neq 0$. Plugging this into our previous equation, we get $\varphi_g(w_0+u_0) = w_0 + u_0 + w_1$, which can be rearranged to become $\varphi_g(u_0) - u_0 = w_0 + w_1 - \varphi_g(w_0)$. The left side of the equation is in $U$ while the right side is in $W$, meaning both sides equal 0. From the left side we conclude $\varphi_g(u_0) = u_0$, meaning $\varphi_g$ fixes $u_0$.
    \end{itemize}   
\end{proof}

\begin{lemma} \thlabel{intersectionOfIrreducibles}
    Let $G$ be a finite group, let $\F$ be a finite field, let $V$ be a finite dimensional $\F$-vector space, and let $\varphi: G \to \GL(V)$ be a representation. Then there exists an $n\in \mathbb{Z}^+$ and $\F$-vector spaces $V_i$ and irreducible representation $\varphi_i: G \to \GL(V_i)$ for $i=1,\dots n$ such that $\pgc(G \ltimes_\varphi V) = \pgc(G \ltimes_{\varphi_1} V_1) \sqcap \cdots \sqcap \pgc(G \ltimes_{\varphi_n} V_n)$.
\end{lemma}
\begin{proof}
    We proceed, somewhat informally, by induction. Assume we are able to find $n\in \mathbb{Z}^+$ and representations $\varphi_1, \cdots, \varphi_n$ and vector spaces $V_1, \cdots, V_n$ such that $\pgc(G \ltimes_\varphi V) = \pgc(G \ltimes_{\varphi_1} V_1) \sqcap \cdots \sqcap \pgc(G \ltimes_{\varphi_n} V_n)$, with the property that $\sum_{i=1}^n \dim(V_i) = \dim(V)$. If all $\varphi_i$ are irreducible, then we are done. Otherwise, assume without loss of generality that $\varphi_1$ is reducible. By \thref{splitRepresentation}, there exists a vector space $\widetilde{V}_1$ with the same dimension as $V_1$ and a decomposable representation $\widetilde{\varphi}_1: G \to \GL(\widetilde{V}_1)$ such that $\pgc(G\ltimes_{\varphi_1} V_1) = \pgc(G\ltimes_{\widetilde{\varphi}_1} \widetilde{V}_1)$. We can then decompose $\widetilde{\varphi}_1: G \to \GL(\widetilde{V}_1)$ into $\widetilde{\varphi}_{11} \oplus \widetilde{\varphi}_{12}: G \to \GL(\widetilde{V}_{11} \oplus \widetilde{V}_{12})$. By \thref{directSumIntersection}, $\pgc(G\ltimes_{\widetilde{\varphi}_1} \widetilde{V}_1) = \pgc(G\ltimes_{\widetilde{\varphi}_{11}} \widetilde{V}_{11}) \sqcap \pgc(G\ltimes_{\widetilde{\varphi}_{12}} \widetilde{V}_{12})$. Thus we have 
    \begin{equation}
        \pgc(G \ltimes_\varphi V) = \pgc(G\ltimes_{\widetilde{\varphi}_{11}} \widetilde{V}_{11}) \sqcap \pgc(G\ltimes_{\widetilde{\varphi}_{12}} \widetilde{V}_{12}) \sqcap \cdots \sqcap \pgc(G \ltimes_{\varphi_n} V_n).
    \end{equation}
    It still holds that $\dim(V_{11}) + \dim(V_{12}) + \cdots + \dim(V_n) = \dim(V)$, so we can repeat our argument. This process must eventually terminate because $\dim(V)$ is finite.
\end{proof}

\begin{lemma} \thlabel{excludeEdge}
    Let $G$ be a finite group, let $p \in \pi(G)$, and let $N$ be a normal $p$-subgroup of $G$ such that $p \in \pi(G/N)$. Then given a prime $r \in \pi(G) \setminus \{p\}$ such that $p-r \notin \pgc(G)$, there exists an elementary abelian $p$-group $A$ and an action $\varphi: G/N \to \Aut(A)$ such that $\pgc(G) \subseteq \pgc(\frac{G}{N} \ltimes_\varphi A)$ and $p-r \notin \pgc(\frac{G}{N} \ltimes_\varphi A)$.
\end{lemma}
\begin{proof}
    If $p-r \notin \pgc(G/N)$, then letting $A$ equal the trivial group works. So we may assume that $p-r \in \pgc(G/N)$. Since $N$ is a $p$-group, it is nilpotent. 
    We thus get a normal series
    \begin{equation}
        G \trianglerighteq N = N_k \trianglerighteq N_{k-1} \trianglerighteq \cdots \trianglerighteq N_1 \trianglerighteq N_0 = \{1\}
    \end{equation}
    such that for all $i \in \{1, \cdots, k\}$ we have $\frac{N_i}{N_{i-1}} = Z(\frac{N_k}{N_{i-1}})$. In this normal series, take the largest $i$ such that $p-r \notin \pgc(\frac{G}{N_i})$. By quotienting every term in the series by $N_i$, we may assume without loss of generality that $i = 0$.

    Let $\widetilde{\varphi}: G/N_1 \to \Aut(N_1)$ be defined by $\widetilde{\varphi}_{gN_1}(n) = g^{-1}ng$. This is well-defined since $N_1$ is abelian. Since $p-r \notin \pgc(G)$ and $p-r \in \pgc(G/N_1)$, we have by \thref{representationSemidirectProduct} that there exists some $g_0N_1 \in G/N_1$ of order $r$ such that $\widetilde{\varphi}_{g_0N_1}$ fixes a nontrivial element $n_0 \in N_1$. 
    
    Let $\varphi: G/N_k \to \Aut(N_1)$ be defined by $\varphi_{gN_k}(n) = g^{-1}ng$. This is well-defined since $N_1 = Z(N_k)$. It is clear that $\varphi_{g_0N_k}$ fixes $n_0$. Additionally, $g_0N_k$ has order $r$ by \thref{orderNotDividing}. By taking a suitable power of $n_0$, we may assume without loss of generality that $o(n_0) = p$. Let $A = \{ n \in N_1 : n^p = 1 \}$ be the maximal elementary abelian subgroup of $N_1$. Since $A$ is characteristic, we have that $G/N_k$ acts on $A$ via $\varphi$. Since $n_0 \in A$ is still fixed by $g_0N_k$, we have by \thref{representationSemidirectProduct} that $p-r \notin \pgc(\frac{G}{N_k} \ltimes_\varphi A)$. 

        It remains to show that $\pgc(G) \subseteq \pgc(\frac{G}{N_k} \ltimes_\varphi A)$. Given distinct primes $q,s \in \pi(G)$ with $q-s \notin \pgc(\frac{G}{N_k} \ltimes_\varphi A)$, \thref{representationSemidirectProduct} gives us two cases:
    \begin{itemize}
        \item $q-s \notin \pgc(G/N_k)$. Then $q-s \notin \pgc(G)$, and we are done.
        \item There exists $gN_k \in G/N_k$ of order $s$ such that $\varphi_{gN_k}$ fixes an element $a \in A$ of order $q$. (Note that this implies $q=p$.) But then $gN_1 \in G/N_1$ is an element of order a multiple of $s$ such that $\widetilde{\varphi}_{gN_1}$ fixes $a$, so by taking a suitable power and applying \thref{representationSemidirectProduct}, we have $q-s \notin \pgc(\frac{G}{N_1} \ltimes_{\widetilde{\varphi}} N_1) = \pgc(G)$.
    \end{itemize}

\end{proof}
\begin{theorem}\thlabel{pGroupExtensionToGraphIntersection}
    Let $T$ be a finite group, let $p \in \pi(T)$, and let $S$ be the set of irreducible representations of $T$ over the field $\F_p$. For $\varphi \in S$, let $V_\varphi$ denote the vector space on which $\varphi$ acts. Then given some graph $\Lambda$, we have that $\Lambda$ is realizable as the prime graph complement of a group of the form $N.T$ where $N$ is a $p$-group if and only if there is some subset $R \subseteq S$ such that $\Lambda = \bigsqcap_{\varphi \in R} \pgc(T \ltimes_\varphi V_\varphi)$.
\end{theorem}
\begin{proof}
    Forward direction: Let $G \cong N.T$ be a group. By \thref{excludeEdge}, for each $r \in \pi(G) \setminus \{p\}$ such that $p-r \notin \pgc(G)$, we get an elementary abelian $p$-group $A_{p-r}$ and an action $\varphi_{p-r}: T \to \Aut(A_{p-r})$ such that $\pgc(G) \subseteq \pgc(T \ltimes_\varphi A_{p-r})$ and $p-r \notin \pgc(T \ltimes_\varphi A_{p-r})$. By \thref{elementaryAbelianExtensionsToRepresentations}, we can view $A_{p-r}$ as the additive group of an $\F_p$-vector space $V_{p-r}$ and $\varphi_{p-r}$ as a representation $T \to \GL(V_{p-r})$. Let
    \begin{equation}
        \widetilde{\varphi} := \bigoplus_{p-r \notin \pgc(G)} \varphi_{p-r} \qquad \text{and} \qquad \widetilde{V} := \bigoplus_{p-r \notin \pgc(G)} V_{p-r}.
    \end{equation}
    By \thref{directSumIntersection} we have that $\pgc(T \ltimes_{\widetilde{\varphi}} \widetilde{V}) = \pgc(G)$, and by 
    \thref{intersectionOfIrreducibles} we have that $\pgc(T \ltimes_{\widetilde{\varphi}} \widetilde{V})$ is equal to the labeled graph product of prime graph complements arising from irreducible representations.

    Backward direction: \thref{directSumIntersection} guarantees that any labeled graph product of prime graph complements arising from irreducible representations can be realized by the direct sum of those representations.
\end{proof}

\begin{theorem}\thlabel{GeneralizedModuleExtensions}
        Let $T$ be a finite simple group, and let $p \in \pi(T)$. For each $\chi \in \operatorname{IBr}_p(T)$, let $A_\chi$ be the set of edges $\{p-q \mid \exists g\in T \textnormal{ s.t. } \operatorname{o}(g) = q \textnormal{ and } \frac{1}{\operatorname{o}(g)}\sum_{x \in \langle g \rangle}\chi(x) > 0\}$. Then given some graph $\Lambda$, we have that $\Lambda$ is realizable as the prime graph complement of a group of the form $N.T$ where $N$ is a $p$-group if and only if there is some subset $Y\subseteq \operatorname{IBr}_p(T)$ such that $\Lambda = \pgc(T)\setminus\mleft(\bigcup_{\chi\in Y} A_\chi\mright)$.
\end{theorem} 
\begin{proof}
        Backward direction: By \thref{graphExistenceFromBrauerTable}, for each $\chi \in Y$ we can find a group $T \ltimes_{\varphi_\chi} V_\chi$ with prime graph complement equal to $\pgc(T) \setminus A_\chi$. From the proof of \thref{graphExistenceFromBrauerTable}, we can choose these $V_\chi$ such that they are all vector spaces over the same field. Applying \thref{directSumIntersection} to $\{ \varphi_\chi \mid \chi \in Y \}$, we get a group with a prime graph complement of $\Lambda$.


    Forward direction: By \thref{pGroupExtensionToGraphIntersection}, it suffices to show that given any irreducible representation $\varphi: T \to \GL(k, \F_p)$ where $k$ is a positive integer, there exists a $Y \subseteq \operatorname{IBr}_p(T)$ such that $\pgc(T \ltimes_\varphi (\F_p)^k) = \pgc(T) \setminus (\bigcup_{\chi \in Y} A_\chi)$. Viewing the codomain of $\varphi$ as matrices, we may view $\varphi$ as a function $\varphi:T \to \GL(k, \F_{p^m})$, where $m$ is large enough so that $\F_{p^m}$ is a splitting field for $T$ (guaranteed by \cite[Corollary 9.10]{IsaacsCharacters}). Note that $\varphi$ may no longer be irreducible. Since the eigenvalues of these matrices have not changed, for all $h\in T$, the matrix $\varphi_h$ fixes a nontrivial element of $(\F_{p^m})^k$ if and only if $\varphi_h$ fixes a nontrivial element of $(\F_p)^k$. Thus, by 
    \thref{representationSemidirectProduct} we have that $\pgc(T \ltimes_\varphi (\F_p)^k) = \pgc(T \ltimes_\varphi (\F_{p^m})^k)$. By \thref{intersectionOfIrreducibles}, this equals 
    $\pgc(T \ltimes_{\varphi_1} (\F_{p^m})^{k_1}) \sqcap \cdots \sqcap \pgc(T \ltimes_{\varphi_n} (\F_{p^m})^{k_n})$ for some irreducible representations $\varphi_1, \cdots, \varphi_n$ and positive integers $k_1, \cdots, k_n$. Letting $\chi_i$ be the irreducible Brauer character corresponding to $\varphi_i$ for $i\in \{1, \cdots, n\}$, this equals $(\pgc(T) \setminus A_{\chi_1}) \sqcap \cdots \sqcap (\pgc(T) \setminus A_{\chi_n})$ by \thref{MaslovaPaperLemma} and \thref{representationSemidirectProduct}. Letting $Y = \{\chi_1, \cdots, \chi_n\}$, this equals $\pgc(T) \setminus (\bigcup_{\chi\in Y} A_\chi)$, completing the proof.
\end{proof}

\section{The Subgraphs Induced by $\pi(T)$ and $\pi(G) \setminus \pi(T)$}\label{sec:subgraphsInduced}
The goal of this section is to prove Theorems \ref{HallSubgroupInducedGraph} and \ref{SolvableHallSubgroupInducedSubgraph}. Given a nonabelian simple group $T$ and a $T$-solvable group $G$, \thref{HallSubgroupInducedGraph} tells us that the subgraph of $\pgc(G)$ induced by any set of primes containing $\pi(T)$ is realizable by a $T$-solvable group, and \thref{SolvableHallSubgroupInducedSubgraph} tells us that the subgraph of $\pgc(G)$ induced by $\pi(G) \setminus \pi(T)$ is realizable by a solvable group. 

\begin{lemma}\thlabel{subgroupOfTSolvable}
    Let $T$ be a nonabelian simple group, and let $G$ be a $T$-solvable group. Let $H$ be a Hall subgroup of $G$ such that $\pi(H) \supseteq \pi(T)$. Then $H$ is $T$-solvable.
\end{lemma}
\begin{proof}
    By the remark under \cite[Lemma 6.4]{Florez}, $H$ has a normal series in which every factor is isomorphic to a Hall $\pi(H)$-subgroup of a chief factor of $G$. Since $\pi(H) \supseteq \pi(T)$, the factors of this series are either elementary abelian or isomorphic to a power of $T$, meaning $H$ is $T$-solvable.
\end{proof}
\begin{lemma}\thlabel{coprimeSubgroupOfTSolvable}
    Let $T$ be a nonabelian simple group, let $G$ be a $T$-solvable group, and let $H$ be a Hall subgroup of $G$ such that $|\pi(H) \cap \pi(T)|\leq 2$. Then $H$ is solvable.
\end{lemma}
\begin{proof}
    By the remark under \cite[Lemma 6.4]{Florez}, $H$ has a normal series in which every factor is isomorphic to a Hall $\pi(H)$-subgroup of a chief factor of $G$. Since $|H|$ shares at most two prime divisors with $|T|$, all factors which are subgroups of a non-solvable chief factor of $G$ have order divisible by at most two primes. Thus all factors are solvable, meaning $H$ is solvable.
\end{proof}

\begin{lemma}\thlabel{2.10Corollary}
    Let $\pi$ be a set of primes and let $G$ be a group such that one of the following holds:
    \begin{itemize}
        \item All nonabelian composition factors of $G$ are $\pi$-groups.
        \item All nonabelian composition factors of $G$ are $\pi'$-groups.
    \end{itemize}
    Then there exists a Hall $\pi$-subgroup $H\leq G$ such that $\pgc(H)$ is isomorphic to the subgraph of $\pgc(G)$ induced by $\pi(H)$.
\end{lemma}
\begin{proof}
    $G$ satisfies the hypothesis of \cite[2.10]{HALLSUB}. Therefore, we can let $H$ be a Hall $\pi$-subgroup of $G$. We must show that given distinct primes $p, r \in \pi(H)$, there exists an element of order $pr$ in $H$ if and only if there exists an element of order $pr$ in $G$. The forward direction is obvious. For the backward direction, let $g \in G$ be an element of order $pr$. Since $\langle g\rangle$ is a $\pi$-subgroup, by \cite[2.10]{HALLSUB} it is contained in some Hall $\pi$-subgroup $H'\leq G$. Since (again by \cite[2.10]{HALLSUB}) all Hall $\pi$-subgroups in $G$ are conjugate, we have that $g$ is conjugate to an element of order $pr$ in $H$.
\end{proof}

\begin{theorem}\thlabel{HallSubgroupInducedGraph}
    Let $T$ be a nonabelian simple group, let $G$ be a $T$-solvable group, and let $\pi_0$ be a set of primes such that $\pi_0 \supseteq \pi(T)$.
    Then there exists a Hall $\pi_0$-subgroup $H \leq G$ such that $H$ is $T$-solvable and $\pgc(H)$ is isomorphic to the subgraph of $\pgc(G)$ induced by $\pi(H)$.
\end{theorem}
\begin{proof}
    Since $T$ is a $\pi_0$-group, $G$ satisfies the hypothesis of \thref{2.10Corollary}. So we get a Hall $\pi_0$-subgroup $H\leq G$ such that $\pgc(H) = \pgc(G)[\pi(H)]$. By \thref{subgroupOfTSolvable}, $H$ is $T$-solvable.
\end{proof}
\begin{theorem}\thlabel{SolvableHallSubgroupInducedSubgraph}
    Let $T$ be a nonabelian simple group, let $G$ be a $T$-solvable group, and let $\pi_0$ be a set of primes such that $\pi_0 \subseteq \pi(G) \setminus \pi(T)$. Then there exists a Hall $\pi_0$-subgroup $H \leq G$ such that $H$ is solvable and $\pgc(H)$ is isomorphic to the subgraph of $\pgc(G)$ induced by $\pi(H)$.
\end{theorem}
\begin{proof}
    Since $T$ is a $\pi_0'$-group, $G$ satisfies the hypothesis of \thref{2.10Corollary}. So we get a Hall $\pi_0$-subgroup $H\leq G$ such that $\pgc(H) = \pgc(G)[\pi(H)]$. By \thref{coprimeSubgroupOfTSolvable}, $H$ is solvable.
\end{proof}

\section{General Lemmas}\label{sec:UsefulLemmas}
This section contains a set of lemmas that we use in the classifications of Sz(8)- and $\PSL(2, 2^5)$-solvable groups. Most are fairly general, and can be applied to $T$-solvable groups for any finite nonabelian simple $T$.

\begin{lemma} \thlabel{solvableSubgraph}
    Let $T$ be a nonabelian simple group, and let $G$ be a $T$-solvable group. Then the subgraph of $\pgc(G)$ induced by $\pi(G) \setminus \pi(T)$ is triangle-free and 3-colorable.
\end{lemma}
\begin{proof}
    We apply \thref{SolvableHallSubgroupInducedSubgraph} with $\pi_0 = \pi(G) \setminus \pi(T)$ to get a solvable subgroup $H \leq G$ such that $\pi(H) = \pi_0$ and $\pgc(H) = \pgc(G)[\pi_0]$. By \thref{solvclass}, $\pgc(H)$ is triangle-free and 3-colorable, so we conclude that $\pgc(G)[\pi_0]$ is triangle-free and 3-colorable.
\end{proof}
We generalize a result from \cite{2022REU}:
\begin{lemma}\thlabel{2.1.2}
    \cite[Corollary 2.4]{2022REU}
    Let $T$ be a finite nonabelian simple group and let $G$ be a strictly $T$-solvable group. Then there exists a subgroup $K\leq G$ such that $K \cong N.T$ where $N$ is solvable and $\pi(K) \supseteq \pi(G) \setminus (\pi(\Aut(T)) \setminus \pi(T))$.
\end{lemma}
\begin{proof}
    The argument is nearly identical to the original, with the only change being that $\pi(T)$ is replaced with $\pi(\Aut(T))$.
\end{proof}

\begin{lemma}\thlabel{Nilpotent}
    Let $G$ be a finite group, let $\pi_0$ be a set of primes in $\pi(G)$, and let $N$ be a solvable normal $\pi_0$-subgroup of $G$. Suppose that there is some $s \in \pi(G) \setminus \pi_0$ such that $s-p \in \pgc(G)$ for all $p\in \pi_0$. Then $N$ is nilpotent.
\end{lemma}
\begin{proof}
    Let $H$ be a Sylow $s$-subgroup of $G$. Then $H$ acts on $N$ by conjugation, and this action is a Frobenius action, otherwise an element of order $s$ in $H$ would commute with an element in $N$ of order some prime in $\pi_0$, a contradiction. Since $N$ is a Frobenius kernel, it is nilpotent. 
\end{proof}


This next lemma provides information about which edges must be absent from the prime graph complement of a $T$-solvable group, given certain conditions on $T$. We show in \thref{disconnect} and \thref{disconnectpsl} that these conditions hold for $T = \Sz(8)$ and $T = \PSL(2,2^5)$, respectively. For the reader's convenience, we reproduce the hypothesis of the following corollary from \cite{2023REU}:

A prime $r \in \pi(T)$ satisfies \cite[Corollary 2.2.6]{2023REU} if $r$ is odd, $r$ does not divide the order of the Schur multiplier of $T$, and an element of order $r$ has nontrivial fixed points in every irreducible representation of every perfect central extension of $T$.
\begin{lemma}\thlabel{Noworrygeneral}
    Let $T$ be a nonabelian finite simple group, and let $Y \subseteq \pi(T)$ be the set of prime divisors of $|T|$ which satisfy the hypothesis of \cite[Corollary 2.2.6]{2023REU}. Furthermore, suppose that for all $s\in\pi(T)\setminus Y$, the Sylow $s$-subgroup of $T$ does not satisfy the Frobenius Criterion (\thref{frobcrit}).
    Let $G$ be a $T$-solvable group with a solvable normal subgroup $N$, and let $R$ be the set of primes $(\pi(G) \setminus \pi(\Aut(T)))\cup (\pi(N) \setminus \pi(T))$. Then $p-r \notin \pgc(G)$ for any $p\in \pi(T)$, $r \in R$.
\end{lemma}
\begin{proof}
    Let $K/N \leq G/N$ be the subgroup given by \thref{2.1.2}. Notice that $R \subseteq \pi(K)$. By applying \cite[Corollary 2.2.6]{2023REU} to $K$, we see that there are no edges between $Y$ and $R$. By applying \cite[Proposition 2.2.2]{2023REU} to $K$, we see that there are no edges between $\pi(T) \setminus Y$ and $R$.
    Thus we conclude that there are no edges between $\pi(T)$ and $R$.
\end{proof}

\begin{lemma} \thlabel{Msubgroup}
    Let $T$ be a nonabelian finite simple group,
    and let $M$ be a subgroup of $\Aut(T)$ containing $\operatorname{Inn}(T)$. Let $p\in \pi(T)$, and let $\pi_0$ be the set of primes $(\pi(M) \setminus \pi(T))\cup \{p\}$. Then $M$ has a solvable strict $\pi_0$-subgroup.
\end{lemma}
\begin{proof}
    By the Schreier conjecture, $M$ is $T$-solvable. We let $H$ be a solvable Hall $(\pi(M) \setminus \pi(T))$-subgroup of $M$ via \thref{SolvableHallSubgroupInducedSubgraph}. $H$ acts coprimely on $\operatorname{Inn}(T) \cong T$ by conjugation. By \cite[Theorem 2.1]{Flavell} there must exist an $H$-invariant Sylow $p$-subgroup of $T$, which we will call $P$. Then $HP$ is a solvable $\pi_0$-subgroup of $M$.
\end{proof}

\begin{corollary}\thlabel{HallTriangles}
    Let $T$ be a nonabelian finite simple group which satisfies the hypothesis of \thref{Noworrygeneral}. Let $G\cong N.M$ where $N$ is solvable and $M$ is a subgroup of $\Aut(T)$ containing $\operatorname{Inn}(T)$. Then given a triangle $A \subseteq \pgc(G)$ such that $A$ has at least one vertex not in $\pi(T)$,
    we have that $V(A) \cap \pi(N) = \emptyset$.
\end{corollary}
\begin{proof}
    By \thref{solvableSubgraph}, at least one of the vertices of $A$ must be in $\pi(T)$. Let $\{a,b,c\}$ be the vertices of $A$, and assume without loss of generality that $a \in \pi(T)$ and $c\notin \pi(T)$. 
    If $c \in \pi(N)$, then we have a contradiction by \thref{Noworrygeneral}. Thus $c\in \pi(M) \setminus \pi(T)$. If $a\in \pi(N)$, then we apply \thref{Msubgroup} to $G/N$ to get a solvable $((\pi(M) \setminus \pi(T))\cup \{b\})$-subgroup $H/N$. Then $H$ is solvable and $A \subseteq \pgc(H)$, a contradiction by \thref{solvclass}.
    Hence $a\not\in \pi(N)$. Now if $b\in \pi(N)$, in the case that $b\notin \pi(T)$ we get a contradiction as for $c$ above, and in the case that
    $b\in \pi(T)$ we get a contradiction as for $a$ above.
\end{proof}

\begin{lemma}\thlabel{SylowTriangles}
    Let $G$ be a finite group, let $N$ be a solvable normal subgroup of $G$, and let $A \subseteq \pgc(G)$ be a triangle. Then, $|V(A)\cap\pi(N)| \leq 1$.
\end{lemma}
\begin{proof}
    Let $V(A) = \{a, b, c\}$. For contradiction, suppose that $a, b \in \pi(N)$. If $c\in \pi(N)$, then $\pgc(N)$ contains a triangle, which is a contradiction by \thref{solvclass}. If $c\notin \pi(N)$, then $c\in \pi(G/N)$. Let $H$ be a Sylow $c$-subgroup of $G/N$. Then the subgroup of $G$ isomorphic to $N.H$ is a solvable group whose prime graph complement contains a triangle, a contradiction by \thref{solvclass}.
\end{proof}

\begin{lemma}\thlabel{solvableExtension}
    Let $G$ be a finite group and let $N$ be a solvable normal subgroup of $G$. Then there exists a normal series $G \trianglerighteq N_k \trianglerighteq N_{k-1} \trianglerighteq \cdots \trianglerighteq N_1 \trianglerighteq N_0 = \{1\}$ such that $N_k = N$ and $\frac{N_i}{N_{i-1}}$ is elementary abelian for $i$ from 1 to $k$.
\end{lemma}
\begin{proof} This is a well-known standard fact.
\end{proof}
Finally, we recall the meaning of the Frobenius criterion.\\

\begin{lemma}\thlabel{twoWithEdgeGeneral}
    Let $T$ be a nonabelian simple group, let $p \in \pi(T)$, let $N$ be a solvable group, let $r \in \pi(N)$, and let $G \cong N.T$. Suppose that the Sylow $p$-subgroups of $T$ do not satisfy the Frobenius Criterion. Then $p-r \notin \pgc(G)$. 
    
\end{lemma}
\begin{proof} This is \cite[Proposition 2.2.2]{2023REU}.
\end{proof}

\section{The Suzuki Group Sz(8)}\label{Sz8}

\begin{figure}[H]
    \centering
    \begin{minipage}{.5\textwidth}
        \centering
        \begin{tikzpicture}
            \coordinate (A) at (0, 1);   
            \coordinate (B) at (1, 1);   
            \coordinate (C) at (1, 0);   
            \coordinate (D) at (0, 0);   
            \draw (B) -- (C);
            \draw (B) -- (D);
            \draw (C) -- (D);
            \draw (A) -- (D);
            \draw (A) -- (B);
            \draw (A) -- (C);
            \foreach \point in {A, B, C, D}
                \draw[draw, fill=white] (\point) circle [radius=0.2cm];
            \foreach \point/\name in {(A)/2, (B)/7, (C)/13, (D)/5}
                \node at \point {\scriptsize\name};
        \end{tikzpicture}
        \caption*{$\pgc(\Sz(8))$}
    \end{minipage}%
    \begin{minipage}{.5\textwidth}
        \centering
        \begin{tikzpicture}
            \coordinate (A) at (0, 1);   
            \coordinate (B) at (1, 1);   
            \coordinate (C) at (1, 0);   
            \coordinate (D) at (0, 0);   
            \coordinate (E) at (1.5, 0.5);
            \draw (B) -- (C);
            \draw (B) -- (D);
            \draw (C) -- (D);
            \draw (A) -- (D);
            \draw (A) -- (B);
            \draw (A) -- (C);
            \draw (B) -- (E);
            \draw (C) -- (E);
            \foreach \point in {A, B, C, D, E}
                \draw[draw, fill=white] (\point) circle [radius=0.2cm];
            \foreach \point/\name in {(A)/2, (B)/7, (C)/13, (D)/5, (E)/3}
                \node at \point {\scriptsize\name};
        \end{tikzpicture}
        \caption*{$\pgc(\Aut(\Sz(8)))$}
    \end{minipage}
\end{figure}

In this section, we classify prime graph complements of $\Sz(8)$-solvable groups. $\Sz(8)$ is the smallest member of the Suzuki simple groups with order $29120 = 2^6\cdot 5\cdot 7 \cdot 13$ \cite{SuzukiOriginal}. Its outer automorphism group, $\operatorname{Out}(\Sz(8))$, is isomorphic to the cyclic group of order 3. Note that this implies $\pi(\Aut(\Sz(8)))\neq\pi(\Sz(8))$.

\begin{fact} \thlabel{Sz8SolvableSubgroups}
    $\Sz(8)$ has solvable strict $\{2,5\}$, $\{2,7\}$, and $\{2,13\}$-subgroups. $\Aut(\Sz(8))$ has solvable strict $\{2,3,5\}$, $\{2,3,7\}$, and $\{2,3,13\}$-subgroups. 
\end{fact}
\subsection*{Graphs on Four Vertices}
The goal of this subsection is to prove \thref{Sz84VertexGraphs}, which states that every graph on 4 vertices except for \hamSandwich is realizable as the prime graph complement of an $\Sz(8)$-solvable group.

\begin{lemma}\thlabel{Sz8ModuleExtensions}\thlabel{BrauerTableInformation}
    Let $p \in \pi(\Sz(8))$, let $N$ be a nontrivial $p$-group, and let $G \cong N.\Sz(8)$. Then we have the following restrictions on $\pgc(G)$:
    \begin{itemize}
        \item If $p \in \{ 5, 7, 13 \}$ then $\pgc(G)$ is isomorphic to \trianglePlusIsolated.
        \item If $p = 2$ then $\pgc(G)$ is isomorphic to \trianglePlusIsolated, \SuzukiThirtyTwoElusiveGraph, or \fourVertexCompleteGraph. \\
        In \SuzukiThirtyTwoElusiveGraph, the vertex of degree three is 13.
    \end{itemize}
      Additionally, for each of the above graphs $\Xi$ there exists a strictly $\Sz(8)$-solvable group $G$ such that $\pgc(G) \cong \Xi$.
\end{lemma}
\begin{proof}

    By applying the methods for finding fixed points of \thref{graphExistenceFromBrauerTable} to the Brauer table of $\Sz(8)$ mod $p$, we gain the tables below. ``Yes'' appears in row $\chi_i$ column $q$ if and only if for $A_{\chi_i}$ as defined in \thref{GeneralizedModuleExtensions} we have $p-q\in A_i$. The restrictions and existence follow from \thref{GeneralizedModuleExtensions}.
    
    
    
    
    \begin{center} 
    \mbox{
    \begin{tabular}{||c c c c||} 
     \hline
     $\chi_i \in \operatorname{IBr}_{2}(\Sz(8))$ & 5 & 7 & 13 \\ [0.5ex] 
     \hline\hline
     $\chi_1$ & Yes & Yes & Yes \\ 
     \hline
     $\chi_2$ through $\chi_4$ & No & No & No\\
     \hline
     $\chi_5$ through $\chi_7$ & Yes & Yes & No\\
     \hline
     $\chi_8$ & Yes & Yes & Yes \\[1ex] 
     \hline
    \end{tabular}
    \quad
    \begin{tabular}{||c c c c||} 
     \hline
     $\chi_i \in \operatorname{IBr}_{7}(\Sz(8))$ & 2 & 5 & 13\\ [0.5ex] 
     \hline\hline
     $\chi_1$ through $\chi_8$ & Yes & Yes & Yes\\[1ex]
     \hline
    \end{tabular}
    }
    \end{center}
    \begin{center}
    \mbox{
    \begin{tabular}{||c c c c||} 
     \hline
     $\chi_i \in \operatorname{IBr}_{5}(\Sz(8))$ & 2 & 7 & 13 \\ [0.5ex] 
     \hline\hline
     $\chi_1$ through $\chi_{10}$ & Yes & Yes & Yes \\ [1ex] 
     \hline
    \end{tabular}
    \quad
    \begin{tabular}{||c c c c||} 
     \hline
     $\chi_i \in \operatorname{IBr}_{13}(\Sz(8))$ & 2 & 5 & 7\\ [0.5ex] 
     \hline\hline
     $\chi_1$ through $\chi_8$ & Yes & Yes & Yes\\[1ex] 
     \hline
    \end{tabular}
    }\vspace{0.2cm}
    \end{center}


\end{proof}

\begin{remark}
        An alternative method for proving \thref{Sz8ModuleExtensions} is to use GAP \cite{GAP} to directly compute the set of irreducible representations of $\Sz(8)$ over the field $\F_p$ with the command \verb|IrreducibleRepresentations(Sz(8),GF(p))|. On our computers, this took several hours to run.
\end{remark}

\begin{remark}
    It is interesting to note that there exists an extension of $\Sz(8)$ by a nontrivial 2-group which has prime graph complement isomorphic to the complete graph, meaning the extension removes no edges. The extension in question is $\Sz(8) \ltimes_\varphi (\F_8)^4$ where $\varphi: \Sz(8) \to \GL(4,8)$ is an irreducible representation listed as ``Matrix representation of dim 4a over GF(8)" in the online {\it Atlas of finite group representations} \cite{ATLAS}.
\end{remark}

\begin{lemma}\thlabel{hamSandwichImpossible}
    There is no $\Sz(8)$-solvable group whose prime graph complement is isomorphic to \hamSandwich.
\end{lemma}
\begin{proof}
    For the sake of contradiction, let $G$ be an $\Sz(8)$-solvable group such that $\pgc(G) \cong \hamSandwich$. 
    By \cite[Lemma 2.1.1]{2023REU} we have that $G\cong N.\Sz(8)$ where $N$ is a solvable group. Applying \thref{solvableExtension}, we get a normal series
    \begin{equation}
        G \trianglerighteq N = N_k \trianglerighteq N_{k-1} \trianglerighteq \cdots \trianglerighteq N_1 \trianglerighteq N_0 = \{1\}
    \end{equation}
    such that $N_i/N_{i-1}$ is elementary abelian for $i$ from 1 to $k$. By \thref{Sz8ModuleExtensions}, $N_k/N_{k-1}$ must be a 2-group. If $N_i/N_{i-1}$ is a 2-group for all $i$ from 1 to $k$, then $G$ is an extension of $\Sz(8)$ by a 2-group, which is a contradiction by \thref{Sz8ModuleExtensions}. Thus, we may assume that there is some $i$ such that $N_i/N_{i-1}$ is a 5, 7, or 13-group. Let $j$ be the largest such $i$. Let $\widetilde{G} = G/N_{j-1}$ and $\widetilde{N} = N/N_{j-1}$, and consider the normal series
    \begin{equation}
        \widetilde{G} \trianglerighteq \widetilde{N} \trianglerighteq \{1\}.
    \end{equation}
    We list the things we know about this series:
    \begin{itemize}
        \item Since $\widetilde{G}$ is a section of $G$, we have that $\pgc(\widetilde{G})$ is either \hamSandwich or \fourVertexCompleteGraph.
        \item $\widetilde{G}/\widetilde{N} \cong \Sz(8)$.
        \item $\pi(\widetilde{N}) = \{2,p\}$ for some $p \in \{5,7,13\}$.
    \end{itemize}
    Notice that $\pgc(\widetilde{G})$ has the property that any two vertices have a shared neighbor. So we are guaranteed some $r \in \{5,7,13\} \setminus \{p\}$ such that $r-p$ and $r-2$ are both in $\pgc(\widetilde{G})$. By \thref{Nilpotent}, $\widetilde{N}$ is nilpotent. Let $\widetilde{M} \trianglelefteq \widetilde{N}$ be the unique Sylow $2$-subgroup of $\widetilde{N}$. Then
    \begin{equation}
        \widetilde{G} \trianglerighteq \widetilde{N} \trianglerighteq \widetilde{M} \trianglerighteq \{1\}.
    \end{equation}
    is a normal series, but then $\widetilde{G}/\widetilde{M}$ is an extension of $\Sz(8)$ by a nontrivial $p$-group, which is a contradiction by \thref{Sz8ModuleExtensions}.
\end{proof}
\begin{lemma} \thlabel{Sz84VertexGraphs}
     Let $\Xi$ be a graph on 4 vertices. Then $\Xi$ is isomorphic to the prime graph complement of some $\Sz(8)$-solvable group if and only if $\Xi$ is not isomorphic to \hamSandwich.
\end{lemma}
\begin{proof}
    If $\Xi$ is triangle-free and 3-colorable, then by \thref{solvclass} it is the prime graph complement of some solvable group. Up to isomorphism, there are four graphs on 4 vertices that are not both triangle-free and three-colorable:
    \begin{itemize}
        \item The complete graph \fourVertexCompleteGraph can be realized by $\Sz(8)$ itself.
        \item The triangle plus one isolated vertex \trianglePlusIsolated can be realized by $\Sz(8) \times C_2$.
        \item The triangle plus one edge \SuzukiThirtyTwoElusiveGraph can be realized by \thref{Sz8ModuleExtensions}. 
        \item The complete graph minus one edge \hamSandwich is impossible by \thref{hamSandwichImpossible}.
    \end{itemize}
\end{proof}

\subsection*{Graphs on Five Vertices}
We begin by stating a structural result:
\begin{theorem}\thlabel{StructureOfSz8Graphs}\thlabel{disconnect}\thlabel{noworry3}
    Let $G$ be a strictly $\Sz(8)$-solvable group. Then $\pgc(G)$ satisfies both of the following:
    \begin{enumerate}
        \item There are no edges $r-p$ for $r\in \pi(\Sz(8))$ and $p\in \pi(G) \setminus (\pi(\Sz(8)) \cup \{3\})$. Furthermore, if $G$ has a solvable normal subgroup whose order is divisible by 3, then $r-3 \notin \pgc(G)$.
        \item The subgraph induced by $\pi(G) \setminus \pi(\Sz(8))$ is triangle-free and 3-colorable.
    \end{enumerate}
\end{theorem}
\begin{proof}
    To prove (1), it suffices to show that the hypothesis of \thref{Noworrygeneral} holds for $T = \Sz(8)$. Through GAP \cite{GAP}, we find that the primes 5, 7, and 13 satisfy \cite[Corollary 2.2.6]{2023REU}, and the Sylow $2$-subgroups of $\Sz(8)$ do not satisfy the Frobenius Criterion. (2) follows from \thref{solvableSubgraph}.
\end{proof}
\begin{definition}
    A rooted graph is a graph in which one of the vertices is distinguished to be the root. Two rooted graphs $A$ and $B$ are isomorphic if there is a graph isomorphism which identifies the root of $A$ with the root of $B$.
\end{definition}
\thref{StructureOfSz8Graphs} tells us that in the prime graph complement of an $\Sz(8)$-solvable group, 3 is the only prime not in $\pi(\Sz(8))$ that can connect to vertices in $\pi(\Sz(8))$. It will be useful to study the possible structures of the subgraph induced by $\pi(\Sz(8)) \cup \{3\}$. To this end, we define $\mathcal{F}$ to be the set of all rooted graphs on 5 vertices that contain at least one triangle, up to isomorphism.
The set $\mathcal{F}$ can be partitioned into 5 disjoint subsets $\mathcal{F}_1, \cdots, \mathcal{F}_5$:
\begin{itemize}
    \item $\mathcal{F}_1$ is the set of graphs in $\mathcal{F}$ in which the root has degree strictly greater than 2.
    \item $\mathcal{F}_2$ is the set of graphs in $\mathcal{F}\setminus \mathcal{F}_1$ for which the subgraph induced by the non-root vertices is isomorphic to \hamSandwich. $\mathcal{F}_2$ contains exactly the following:
    \begin{center}
        \hamSandwichPlusIsolated, \hamsandwichwitharm, \hamSandwichWithOtherArm, \housenobrim, \nobutt, \fan.
    \end{center}
    \item $\mathcal{F}_3$ is the set of graphs in $\mathcal{F}\setminus \mathcal{F}_2$ in which the root is isolated. In particular, note that $\mathcal{F}_3$ does not contain \smallHamSandwichPlusIsolated.
        
    \item $\mathcal{F}_4$ is the set containing the following 7 graphs: 
    \begin{center}
        \house, \northStar, \codart, \happyFace, \balloon, \dart, \triangleWithShortTail.
    \end{center}
    \item $\mathcal{F}_5$ is the set containing the following 13 graphs (arranged below in an order which we will see is convenient):
    \begin{center}
        \emptyhouse, \triangleWithTail, \trianglestick, \bowtiebruh, \\
        \completeGraphWithArm, \triangleIsolatedEdge, \zigZag, \\
        \cricket, \triangleWithTwoTails, \bullcaseone, \misshapenhouse, \badDart, \\
        \bull.
    \end{center}
\end{itemize}
To verify this list is exhaustive, see Appendix \ref{AppendixA}.
\begin{definition} \thlabel{rootedGraphIsomorphism}
    Given a group $G$ with $\pi(G) = \{2,3,5,7,13\}$ and a rooted graph $\Lambda \in \mathcal{F}$, it is understood that when we say $\pgc(G) \cong \Lambda$, we require the isomorphism to map the vertex 3 to the root.
\end{definition}
\begin{definition}\thlabel{realizableDef}
    Given a set of rooted graphs $\mathcal{H}\subseteq \mathcal{F}$, we say that $\mathcal{H}$ is \emph{realizable} if for each $\Lambda \in \mathcal{H}$, there exists an $\Sz(8)$-solvable $\{2,3,5,7,13\}$-group $G$ such that $\pgc(G) \cong \Lambda$ in the sense of \thref{rootedGraphIsomorphism}.
\end{definition}
The following theorem tells us that studying which graphs in $\mathcal{F}$ are realizable gives us information about the possible subgraphs induced by $\pi(\Sz(8)) \cup \{3\}$.
\begin{theorem} \thlabel{sufficientOn5}
    Let $\Lambda \in \mathcal{F}$ be a rooted graph on 5 vertices which is not realizable. Then there exists no $\Sz(8)$-solvable group $G$ such that $\pgc(G)[\{2,3,5,7,13\}] \cong \Lambda$.
\end{theorem}
\begin{proof}
    Suppose there did exist such a group $G$. Then by \thref{HallSubgroupInducedGraph}, we can take a Hall $\{2,3,5,7,13\}$-subgroup of $G$ which is $\Sz(8)$-solvable and has prime graph complement isomorphic to $\Lambda$, a contradiction.
\end{proof}

Our goal for the remainder of this subsection is to prove \thref{Realizable5VertexGraphs}, which states that the graphs in $\mathcal{F}_3 \cup \mathcal{F}_4$ are the only realizable graphs in $\mathcal{F}$.

\begin{lemma} \thlabel{AutSz8OnTop}
    If a rooted graph $\Lambda \in \mathcal{F} \setminus \mathcal{F}_3$ is realizable by a group $G$, then all of the following hold:
    \begin{itemize}
        \item $G \cong N.\Aut(\Sz(8))$ where $N$ is solvable.
        \item $3$ does not divide $|N|$.
        \item If a vertex in $\Lambda$ is adjacent to the root, then that vertex corresponds to either 7 or 13 in $\pgc(G)$.
    \end{itemize}
\end{lemma}
\begin{proof}
    Let $\Lambda$ be a rooted graph in $\mathcal{F}\setminus \mathcal{F}_3$. Since $\Lambda$ has a triangle, $G$ must be strictly $\Sz(8)$-solvable. By \cite[Lemma 2.1.1]{2023REU} we have that $G \cong N.M$ for some solvable $N$ and some $M\leq \Aut(\Sz(8))$ with $M \geq \operatorname{Inn}(\Sz(8)) \cong \Sz(8)$. Since 3 is not isolated, we have by \thref{noworry3} that 3 does not divide $|N|$. Thus, 3 must divide $M$, meaning $M = \Aut(\Sz(8))$. Finally, since $\Aut(\Sz(8))$ is a section of $G$ and the edges $3-2$ and $3-5$ are not in $\pgc(\Aut(\Sz(8)))$, we have that only 7 and 13 can be adjacent to 3.
\end{proof}

\begin{proposition}\thlabel{degreeGreaterThan2}
    $\mathcal{F}_1$ is not realizable.
\end{proposition}
\begin{proof}
    Let $\Lambda$ be a rooted graph in $\mathcal{F}_1$, and suppose for contradiction that $G$ realizes it. By \thref{AutSz8OnTop}, we have that the vertices adjacent to the root in $\Lambda$ correspond to either 7 or 13. In particular, this means there can be at most two vertices adjacent to the root in $\Lambda$, contradicting the assumption that $\Lambda \in \mathcal{F}_1$.
\end{proof}
\begin{proposition}\thlabel{noExtraHams}
    $\mathcal{F}_2$ is not realizable.
\end{proposition}
\begin{proof}
    Suppose there exists some group $G$ that realizes a graph in $\mathcal{F}_2$.
    We can take a Hall $\{2, 5, 7, 13\}$-group of $G$ via \thref{HallSubgroupInducedGraph}. This group is an $\Sz(8)$-solvable group that realizes \hamSandwich which contradicts \thref{hamSandwichImpossible}.
\end{proof}

\begin{proposition} \thlabel{S3Possible}
    $\mathcal{F}_3$ is realizable.
\end{proposition}
\begin{proof}
    Consider the groups from \thref{Sz84VertexGraphs} and take their direct product with $C_3$.
\end{proof}

\begin{proposition}\thlabel{F4Realizable}
    $\mathcal{F}_4$ is realizable.
\end{proposition}
\begin{proof}
    \phantom{.}
    \begin{itemize}
    \item The graph \house can be realized by $\Aut(\Sz(8))$.
    \item The graph \northStar can be realized by $\Aut(\Sz(8))\times C_2$.
    \item The graph \codart can be realized by $\Aut(\Sz(8))\times C_{13}$.
    \item  The graph \happyFace can be realized by $\Aut(\Sz(8)) \times C_{10}$.
    \item The graph \balloon can be realized by $(\Sz(8) \times C_7) \rtimes C_3$ where $C_3$ acts on $\Sz(8)$ as in $\Aut(\Sz(8))$, and Frobeniusly on $C_7$.
    \item The graph \dart can be realized by $\Aut(\Sz(8)) \ltimes_\varphi (\F_2)^{48}$, where $\varphi: \Aut(\Sz(8)) \to \GL(48,2)$ is the irreducible representation of $\Aut(\Sz(8))$ in \cite{ATLAS}. 
    \item The graph \triangleWithShortTail can be realized by the direct product of the previous group with $C_5$. 
    \end{itemize}
\end{proof}

\begin{lemma}\thlabel{twoWithEdge}
    Let $N$ be solvable, let $r \in \pi(N)$, and let $G \cong N.\Sz(8)$. Then $2-r \notin \pgc(G)$.
    
\end{lemma}
\begin{proof}
    The Sylow 2-subgroups of $\Sz(8)$ do not satisfy the Frobenius Criterion. Apply \thref{twoWithEdgeGeneral}.
\end{proof}

\begin{proposition}\thlabel{triangleWithThreeFamily}
    Let $\Lambda \in \mathcal{F}_5$ be a rooted graph such that there exists a triangle containing the root. Additionally, suppose that the two vertices not part of this triangle are connected by an edge and each have degree strictly less than 3. The possibilities for $\Lambda$ are exactly the following:
    \begin{center}
        \emptyhouse, \triangleWithTail, \trianglestick, \bowtiebruh.
    \end{center}
    Then $\Lambda$ is not realizable.
\end{proposition}
\begin{proof}
    Suppose for contradiction that $\Lambda$ is realized by an $\Sz(8)$-solvable group $G$. By \thref{AutSz8OnTop}, we have that $G \cong N.\Aut(\Sz(8))$ where $N$ is solvable, with 3 not dividing $|N|$, and additionally that the vertices of the triangle are $\{3, 7, 13\}$. By \thref{HallTriangles}, 7 and 13 do not divide $|N|$. Thus, $N$ is a $\{2,5\}$-group. Since the vertices 2 and 5 both have missing edges, we have by \thref{divisors} that $N$ is a strict $\{2,5\}$-group. Thus, by \thref{twoWithEdge} we have that $2-5 \notin \Lambda$, a contradiction.
\end{proof}

\begin{proposition}\thlabel{triangleWithoutThreeFamily}
    Let $\Lambda \in \mathcal{F}_5$ be a rooted graph such that the root has degree 1. Additionally, suppose that there exists a triangle disjoint from the root and its neighbor. The possibilities for $\Lambda$ are exactly the following:
    \begin{center}
        \completeGraphWithArm, \triangleIsolatedEdge, \zigZag.
    \end{center}
    Then $\Lambda$ is not realizable.
\end{proposition}
\begin{proof}
    Suppose for contradiction that $\Lambda$ is realized by an $\Sz(8)$-solvable group $G$. By \thref{AutSz8OnTop}, we have that $G \cong N.\Aut(\Sz(8))$ where $N$ is solvable, with 3 not dividing $|N|$. Let $r$ be the vertex adjacent to the root. $r$ must then correspond to 7 or 13. Let $v$ be the vertex corresponding to $\{7,13\} \setminus \{r\}$. Since $v-3 \notin \Lambda$, we have by \thref{divisors} that $v$ divides $|N|$. By \thref{Sz8SolvableSubgroups} we can let $H$ be a strict $\{2,5\}$-subgroup of $\Aut(\Sz(8))$. The subgroup of $G$ isomorphic to $N.H$ is solvable, yet its prime graph complement contains a triangle, a contradiction by \thref{solvclass}.
\end{proof}

\begin{proposition} \thlabel{ThirdFamily}
    Let $\Lambda \in \mathcal{F}_5$ be one of the following graphs:
    \begin{center}
        (1)\cricket, (2)\triangleWithTwoTails, (3)\bullcaseone, (4)\misshapenhouse, (5)\badDart.
    \end{center}
    Then $\Lambda$ is not realizable.
\end{proposition}
\begin{proof}
    For contradiction, assume there exists some $\Sz(8)$-solvable group $G$ that realizes $\Lambda$. We can apply \thref{AutSz8OnTop} to see that $G \cong N.\Aut(\Sz(8))\cong N.\Sz(8).C_3$ with 3 not dividing $|N|$. Let $H=N.\Sz(8)$ and observe that $H\trianglelefteq G$. 
    
    We show that $|\pi(N)| \neq 1$: If $|N|$ were a power of 5, 7, or 13, we would get a contradiction by \thref{Sz8ModuleExtensions}. If $|N|$ were a power of 2, then for graphs (1) through (4) we would get a contradiction by \thref{EdgeAdd}, and for graph (5) we would get a contradiction by \thref{Sz8ModuleExtensions}.
    
    
    We show that $|\pi(N)| \leq 2$. For graphs (2) through (5) this follows from \thref{SylowTriangles}, and for graph (1) this follows from \thref{HallTriangles}. 
    
    So we can assume that $|\pi(N)| = 2$. By \thref{Sz8ModuleExtensions}, one of the prime divisors of $|N|$ must be 2. So let $p \neq 2$ be the other prime divisor of $|N|$. By \thref{twoWithEdge}, we have that $2-p \notin \pgc(H)$. Notice that all graphs listed above have the property that any two non-root nonadjacent vertices have a shared neighbor. Therefore, there exists a vertex $q$ such that $p-q$ and $2-q$ are both in $\pgc(G)$. By \thref{Nilpotent}, we find that $N$ is nilpotent. Let $N_0$ be the unique Sylow 2-subgroup of $N$. Then $H/N_0$ is an extension of $\Sz(8)$ by a $p$-group, and we obtain a contradiction by \thref{Sz8ModuleExtensions}.
\end{proof}

\begin{lemma}\thlabel{Bull}
    The graph \bull is not realizable.
\end{lemma}
\begin{proof}
    Suppose for contradiction that the above graph is realized by an $\Sz(8)$-solvable group $G$. By \thref{AutSz8OnTop}, we have that $G \cong N.\Aut(\Sz(8))$ where $N$ is solvable, with 3 not dividing $|N|$. Let $p$ be the vertex such that $5-p\in\pgc(G)$.
    
    By \thref{HallTriangles}, $N$ is a $\{2,5\}$-group. By \thref{divisors}, $N$ is a strict $\{2,5\}$-group. By taking a section of the normal series given by \thref{solvableExtension}, we reduce down to two cases: 
    \begin{enumerate}
        \item $N \cong K_2.K_5$ where $K_2$ is an elementary abelian $2$-group and $K_5$ is a 5-group.
        \item $N \cong K_5.K_2$ where $K_2$ is a 2-group and $K_5$ is an elementary abelian 5-group.
    \end{enumerate}
    If (1) holds then we immediately come to a contradiction since we can take the section $K_5.\Sz(8)$ of $G$ and apply \thref{Sz8ModuleExtensions} to see that the vertex $5$ must be isolated. 
    
    So assume that (2) holds. Let $G_0$ be the section of $G$ isomorphic to $K_5.K_2.\Sz(8)$. By \thref{split}, we may assume that $G_0 \cong K_5\rtimes H$, where $H \cong K_2.\Sz(8)$. We notice that $C_H(K_5)\leq K_2$ or else $C_H(K_5)K_2/K_2$ is a nontrivial normal subgroup of $H/K_2 \cong \Sz(8)$, meaning it is isomorphic to $\Sz(8)$, implying that $C_H(K_5)$ has an element of order $p$, contradicting the fact that $5-p \in \pgc(G_0)$. Additionally, $C_H(K_5) \neq K_2$ or else $N \cong K_2 \times K_5$, so $G_0 \cong K_2.K_5.\Sz(8)$, which is a contradiction by \thref{Sz8ModuleExtensions}. 
    By modding out by $C_H(K_5)$, we may assume without loss of generality that $K_2$ acts faithfully on $K_5$. Let $K_p$ be a Sylow $p$-subgroup of $H$. Since $5-p \in \pgc(G)$, we have that the action of $K_p$ on $K_5$ is Frobenius and therefore faithful. We conclude that the action of $K_2K_p$ on $K_5$ is faithful. 

    Now we show that $K_p$ acts nontrivially on $K_2$ by conjugation. If the action were trivial, we would have that the action $\varphi: H/K_2 \to \Aut(Z(K_2))$ defined by $\varphi_{hK_2}(k) = h^{-1}kh$ has nontrivial kernel. But then the kernel is all of $\Sz(8)$, meaning the vertex $2$ is isolated in $\pgc(H)$, a contradiction.

    Since neither 7 nor 13 is a Fermat prime, by \cite[Lemma 2.1.6]{2023REU} we have that $K_p$ has a nontrivial fixed point on $K_5$, which contradicts the assumption that $5-p \in \pgc(G)$.
\end{proof}

\begin{proposition}\thlabel{Realizable5VertexGraphs}
    The graphs in $\mathcal{F}$ which are realizable are exactly $\mathcal{F}_3 \cup \mathcal{F}_4$.
\end{proposition}
\begin{proof}
    By \thref{degreeGreaterThan2}, $\mathcal{F}_1$ is not realizable. By \thref{noExtraHams}, $\mathcal{F}_2$ is not realizable. By Propositions \ref{triangleWithThreeFamily}, \ref{triangleWithoutThreeFamily}, \ref{ThirdFamily}, and \ref{Bull}, $\mathcal{F}_5$ is not realizable. By \thref{S3Possible}, $\mathcal{F}_3$ is realizable. By \thref{F4Realizable}, $\mathcal{F}_4$ is realizable.

\end{proof}

\subsection*{Classification}
\begin{theorem}\thlabel{Sz8Classification}
Given a graph $\Xi$, we have that $\Xi$ is isomorphic to the prime graph complement of some $\Sz(8)$-solvable group if and only if one of the following is true: 
\begin{enumerate}
    \item $\Xi$ is triangle-free and 3-colorable.
    \item There exists a subset $X = \{a,b,c,d\} \subseteq V(\Xi)$ equal to its own closed neighborhood in $\Xi$, where the subgraph induced by $X$ is not \hamSandwich, such that $X$ contains at least one triangle and $\Xi \setminus X$ is triangle-free and 3-colorable.
    \item There exists a vertex $e \in V(\Xi)$ and a subset $X = \{a,b,c,d\} \subseteq V(\Xi)$ whose closed neighborhood $N[X]$ equals $\{a,b,c,d,e\}$. Furthermore, the subgraph induced by $N[X]$ is isomorphic to one of the graphs in $\mathcal{F}_4$, listed below: 
    \begin{center}
        \house, \northStar, \codart, \happyFace, \balloon, \dart, \triangleWithShortTail.
    \end{center}
    Additionally, $e$ must correspond to the white vertex and $\Xi \setminus X$ must be triangle-free and 3-colorable.
\end{enumerate}
\end{theorem}
\begin{proof}
    Forward direction: Let $G$ be an $\Sz(8)$-solvable group which realizes $\Xi$. If $G$ is solvable, then $\pgc(G)$ satisfies (1) by \thref{solvclass}. If $G$ is strictly $\Sz(8)$-solvable and does not satisfy (1), then $\pgc(G)[\{2,3,5,7,13\}] \in \mathcal{F}_3 \cup \mathcal{F}_4$ by \thref{sufficientOn5} and \thref{Realizable5VertexGraphs}.  By \thref{disconnect}, if $\pgc(G)[\{2,3,5,7,13\}] \in \mathcal{F}_3$ then $\pgc(G)$ satisfies (2), and if $\pgc(G)[\{2,3,5,7,13\}] \in \mathcal{F}_4$ then $\pgc(G)$ satisfies (3).

    Backward direction: We now turn to showing that given a graph $\Xi$ that satisfies (1), (2), or (3), there exists a group $G$ such that $\pgc(G)\cong\Xi$. We proceed by cases. If $\Xi$ satisfies (1), then there exists a solvable group $G$ such that $\pgc(G)\cong \Xi$ by \thref{solvclass}. If $\Xi$ satisfies (2), there exists an $\Sz(8)$-solvable group $H$ such that $\pgc(H)\cong \Xi[X]$ by \thref{Sz84VertexGraphs}. By the construction given in the proof of \cite[Theorem 2.8]{2015REU}, there exists a solvable group $N$ such that $(|N|, |H|) = 1$ and $\pgc(N)\cong\Xi \setminus X$. Then we have $\pgc(H\times N)\cong \Xi$.

    We now turn to the final case, the case where $\Xi$ satisfies (3). By \thref{F4Realizable}, there exists an $\Sz(8)$-solvable group of the form $H \rtimes C_3$ which realizes $\Xi[\{a,b,c,d,e\}]$ such that the vertex 3 is identified with $e$, and 3 does not divide $|H|$. By the construction given in \cite[Theorem 2.8]{2015REU}, there is a solvable group of the form $N \rtimes C_3$ which realizes $\Xi\setminus X$ such that 3 is identified with $e$, and $(|N|, 3|H|) = 1$. Then $(N \times H) \rtimes C_3$ realizes $\Xi$.
\end{proof}

\section{The Projective Special Linear Group PSL(2,$\mathbf{2^5}$)}\label{sec:PSL}
\begin{figure}[H]
    \centering
    \begin{minipage}{.5\textwidth}
        \centering
        \begin{tikzpicture}
            \coordinate (A) at (0, 1);   
            \coordinate (B) at (1, 1);   
            \coordinate (C) at (1, 0);   
            \coordinate (D) at (0, 0);   
            \draw (B) -- (C);
            \draw (C) -- (D);
            \draw (A) -- (D);
            \draw (A) -- (B);
            \draw (A) -- (C);
            \foreach \point in {A, B, C, D}
                \draw[draw, fill=white] (\point) circle [radius=0.2cm];
            \foreach \point/\name in {(A)/2, (B)/11, (C)/31, (D)/3}
                \node at \point {\scriptsize\name};
        \end{tikzpicture}
        \caption*{$\pgc(\PSL(2,2^5))$}
    \end{minipage}%
    \begin{minipage}{.5\textwidth}
        \centering
        \begin{tikzpicture}
            \coordinate (A) at (0, 1);   
            \coordinate (B) at (1, 1);   
            \coordinate (C) at (1, 0);   
            \coordinate (D) at (0, 0);   
            \coordinate (E) at (1.5, 0.5);
            \draw (B) -- (C);
            \draw (C) -- (D);
            \draw (A) -- (D);
            \draw (A) -- (B);
            \draw (A) -- (C);
            \draw (B) -- (E);
            \draw (C) -- (E);
            \foreach \point in {A, B, C, D, E}
                \draw[draw, fill=white] (\point) circle [radius=0.2cm];
            \foreach \point/\name in {(A)/2, (B)/11, (C)/31, (D)/3, (E)/5}
                \node at \point {\scriptsize\name};
        \end{tikzpicture}
        \caption*{$\pgc(\Aut(\PSL(2,2^5)))$}
    \end{minipage}
\end{figure}

We now turn to $\PSL(2, 2^5)$, the smallest projective special linear group with $\pi(\Aut(\PSL(2, 2^5)))\neq\pi(\PSL(2, 2^f))$ \cite{2023REU}. It has order $32736=2^5\cdot 3\cdot 11\cdot 31$ and its outer automorphism group is isomorphic to the cyclic group of order 5. Since this section is very similar in structure to the previous section, we omit some of the exposition.
Each theorem is preceded by a reference to its analogue in Section \ref{Sz8} when applicable. Even for theorems that have such a reference, the arguments used to prove them sometimes follow different logic. These differences mainly arise from the fact that $\pgc(\PSL(2,2^5))$ is missing the $3-11$ edge. In particular, the way that the set of graphs $\mathcal{F}$ is partitioned is different enough from Section \ref{Sz8} that Propositions \ref{PSLimpossible3} and \ref{PSLimpossible4} have no direct analogue. 
\begin{fact}\thlabel{PSL232SolvableSubgroups}
    (\ref{Sz8SolvableSubgroups})
    $\PSL(2,2^5)$ has solvable strict $\{2,3,11\}$ and $\{2,31\}$-subgroups.
    $\Aut(\PSL(2,2^5))$ has solvable strict $\{2,3,5,11\}$ and $\{2,5,31\}$-subgroups.
\end{fact}
\subsection*{Graphs on 4 Vertices}

\begin{lemma}\thlabel{psl232ModuleExtensions}
    (\ref{Sz8ModuleExtensions}) Let $p \in \pi(\operatorname{PSL}(2,2^5))$, let $N$ be a nontrivial $p$-group, and let $G \cong N.\operatorname{PSL}(2,2^5)$. Then we have the following restrictions on $\pgc(G)$:
    \begin{itemize}
        \item If $p \in \{ 3, 11\}$ then $\pgc(G)$ is isomorphic to \trianglePlusIsolated.
        \item If $p = 31$, then $\pgc(G)$ is isomorphic to \brokenFrame.
        \item If $p = 2$ then $\pgc(G)$ is isomorphic to \SuzukiThirtyTwoElusiveGraph, \hamSandwich, \fletching, or \brokenFrame. \\
        In \SuzukiThirtyTwoElusiveGraph, the vertex of degree three is 31 and the 
  vertex of degree one is 3.
    \end{itemize}
    Additionally, for each of the above graphs $\Xi$ there exists a strictly $\PSL(2, 2^5)$-solvable group $G$ such that $\pgc(G) \cong \Xi$.
\end{lemma}
\begin{proof}
    We will denote $\PSL(2, 2^5)$ as $H$ for brevity. By applying the methods for finding fixed points of \thref{graphExistenceFromBrauerTable} to the Brauer table of $H$  mod $q$, we gain the tables below. ``Yes'' appears in row $\chi_i$ column $q$ if and only if there is some $A_i$ from \thref{GeneralizedModuleExtensions} such that $p-q\in A_i$. The restrictions and existence follow from \thref{GeneralizedModuleExtensions}.
 \begin{center} 
    \mbox{
    \begin{tabular}{||c c c c||} 
     \hline
     $\chi_i \in \operatorname{IBr}_{2}(H)$ & 3 & 11 & 31 \\ [0.5ex] 
     \hline\hline
     $\chi_1$ & Yes & Yes & Yes \\ 
     \hline
     $\chi_2$ through $\chi_6$ & No & No & No\\
     \hline
     $\chi_7$ through $\chi_{21}$ & Yes & No & No\\
     \hline
     $\chi_{22}$ through $\chi_{31}$ & Yes & Yes & No \\ 
     \hline
     $\chi_{22}$ through $\chi_{31}$ & Yes & Yes & Yes \\[1ex] 
     \hline
    \end{tabular}
    \quad
    \begin{tabular}{||c c c c||} 
     \hline
     $\chi_i \in \operatorname{IBr}_{11}(H)$ & 2 & 3 & 31\\ [0.5ex] 
     \hline\hline
     $\chi_1$ through $\chi_{18}$ & Yes & Yes & Yes\\[1ex]
     \hline
    \end{tabular}
    }
    \end{center}
    \begin{center}
    \mbox{
    \begin{tabular}{||c c c c||} 
     \hline
     $\chi_i \in \operatorname{IBr}_{3}(H)$ & 2 & 11 & 31 \\ [0.5ex] 
     \hline\hline
     $\chi_1$ through $\chi_{22}$ & Yes & Yes & Yes \\ [1ex] 
     \hline
    \end{tabular}
    \quad
    \begin{tabular}{||c c c c||} 
     \hline
     $\chi_i \in \operatorname{IBr}_{31}(H)$ & 2 & 3 & 11\\ [0.5ex] 
     \hline\hline
     $\chi_1$ through $\chi_{18}$ & Yes & Yes & Yes\\[1ex] 
     \hline
    \end{tabular}
    }\vspace{0.2cm}
    \end{center}
\end{proof}

\begin{lemma}\thlabel{completeGraphImpossible}
    (\ref{hamSandwichImpossible}) There is no $\PSL(2,2^5)$-solvable group whose prime graph complement is isomorphic to \fourVertexCompleteGraph.
\end{lemma}
\begin{proof}
    Any strictly $\PSL(2,2^5)$-solvable $K_4$ group will have $\PSL(2,2^5)$ as a section, and will therefore be a subgraph of $\pgc(\PSL(2,2^5)) \cong \hamSandwich$. No solvable group can realize \fourVertexCompleteGraph by \thref{solvclass}.
\end{proof}

\begin{corollary} \thlabel{PSL2324VertexGraphs}
    (\ref{Sz84VertexGraphs}) Let $\Xi$ be a graph on 4 vertices. Then $\Xi$ is isomorphic to the prime graph complement of some $\PSL(2,2^5)$-solvable group if and only if $\Xi$ does not equal \fourVertexCompleteGraph.
\end{corollary}
\begin{proof}
    If $\Xi$ is triangle-free and 3-colorable, then by \thref{solvclass} it is the prime graph complement of some solvable group. Up to isomorphism, there are four graphs on 4 vertices that are not both triangle-free and three-colorable:
    \begin{itemize}
        \item The complete graph minus one edge \hamSandwich can be realized by $\PSL(2,2^5)$.
        \item The triangle plus one isolated vertex \trianglePlusIsolated can be realized by $\PSL(2,2^5) \times C_3$.
        \item The triangle plus one edge \SuzukiThirtyTwoElusiveGraph can be realized by \thref{psl232ModuleExtensions}.
        \item The complete graph \fourVertexCompleteGraph is impossible by \thref{completeGraphImpossible}.
    \end{itemize}
\end{proof}

\subsection*{Graphs on Five Vertices}
\begin{theorem}\thlabel{StructureOfPSL232Graphs}\thlabel{disconnectpsl}\thlabel{noworry5}
    (\ref{StructureOfSz8Graphs}) Let $G$ be a strictly $\PSL(2,2^5)$-solvable group. Then $\pgc(G)$ satisfies both of the following:
    \begin{enumerate}
        \item There are no edges $r-p$ for $r\in \pi(\PSL(2,2^5))$ and $p\in \pi(G) \setminus (\pi(\PSL(2,2^5)) \cup \{5\})$. Furthermore, if $G$ has a solvable normal subgroup whose order is divisible by 5, then $r-5 \notin \pgc(G)$.
        \item The subgraph induced by $\pi(G) \setminus \pi(\PSL(2,2^5))$ is triangle-free and 3-colorable.
    \end{enumerate}
\end{theorem}
\begin{proof}
    To prove (1), it suffices to show that the hypothesis of \thref{Noworrygeneral} holds for $T = \PSL(2,2^5)$. Through GAP \cite{GAP}, we find that the primes 3, 11, and 31 satisfy \cite[Corollary 2.2.6]{2023REU}, and the Sylow $2$-subgroups of $\PSL(2,2^5)$ do not satisfy the Frobenius Criterion. (2) follows from \thref{solvableSubgraph}.
\end{proof}

Let $\mathcal{F}$ to be the set of all rooted graphs on 5 vertices that contain at least one triangle, up to isomorphism.
The set $\mathcal{F}$ can be partitioned into 5 disjoint subsets $\mathcal{F}_1, \cdots, \mathcal{F}_5$:
\begin{itemize}
    \item $\mathcal{F}_1$ is the set of graphs in $\mathcal{F}$ in which the root has degree strictly greater than 2.
    \item $\mathcal{F}_2$ is the set of graphs in $\mathcal{F}\setminus \mathcal{F}_1$ for which the subgraph induced by the non-root vertices is isomorphic to \fourVertexCompleteGraph. $\mathcal{F}_2$ contains exactly the following:
    \begin{center}
        \completeGraphIsolatedVertex, \completeGraphWithArm, \house
    \end{center}
    \item $\mathcal{F}_3$ is the set of graphs in $\mathcal{F}\setminus \mathcal{F}_2$ in which the root is isolated. In particular, note that $\mathcal{F}_3$ does not contain \smallCompleteGraphIsolatedVertex.

    \item $\mathcal{F}_4$ is the set containing the following 8 graphs: 
    \begin{center}
        \fan, \northStar, \triangleWithShortTail, \codart, \happyFace, \balloon, \dart, \cricket.
    \end{center}
    \item $\mathcal{F}_5$ is the set containing the following 15 graphs (arranged below in an order which we will see is convenient):
    \begin{center}
        \emptyhouse, \triangleWithTail, \trianglestick, \bowtiebruh, \badDart, \\ \triangleIsolatedEdge, \zigZag, \hamsandwichwitharm, \\
        \triangleWithTwoTails, \hamSandwichWithOtherArm, \\
        \housenobrim, \nobutt, \bullcaseone, \misshapenhouse, \\
        \bull. \\

    \end{center}
\end{itemize}
\begin{definition} \thlabel{rootedGraphIsomorphismPSL}
    (\ref{rootedGraphIsomorphism}) Given a group $G$ with $\pi(G) = \{2,3,5,11,31\}$ and a rooted graph $\Lambda \in \mathcal{F}$, it is understood that when we say $\pgc(G) \cong \Lambda$, we require the isomorphism to map the vertex 5 to the root.
\end{definition}
\begin{definition}\thlabel{realizableDefPSL}
    (\ref{realizableDef}) Given a set of rooted graphs $\mathcal{H}\subseteq \mathcal{F}$, we say that $\mathcal{H}$ is \emph{realizable} if for each $\Lambda \in \mathcal{H}$, there exists a $\PSL(2,2^5)$-solvable $\{2,3,5,11,31\}$-group $G$ such that $\pgc(G) \cong \Lambda$ in the sense of \thref{rootedGraphIsomorphismPSL}.
\end{definition}

\begin{theorem} \thlabel{sufficientOn5PSL}
    (\ref{sufficientOn5}) Let $\Lambda \in \mathcal{F}$ be a rooted graph on 5 vertices which is not realizable. Then there exists no $\PSL(2,2^5)$-solvable group $G$ such that $\pgc(G)[\{2,3,5,11,31\}] \cong \Lambda$.
\end{theorem}
\begin{proof}
    Suppose there did exist such a group $G$. Then by \thref{HallSubgroupInducedGraph}, we can take a Hall $\{2,3,5,11,31\}$-subgroup of $G$ which is $\PSL(2,2^5)$-solvable and has prime graph complement isomorphic to $\Lambda$, a contradiction.
\end{proof}

\begin{lemma} \thlabel{AutPSL232OnTop}
    (\ref{AutSz8OnTop}) If a rooted graph $\Lambda \in \mathcal{F} \setminus \mathcal{F}_3$ is realizable by a group $G$, then all of the following hold:
    \begin{itemize}
        \item $G \cong N.\Aut(\PSL(2,2^5))$ where $N$ is solvable.
        \item Neither $5$ nor $31$ divide $|N|$.
        \item If a vertex $v \in \Lambda$ is adjacent to the root, then $v$ corresponds to either 11 or 31 in $\pgc(G)$. If $v$ is the only vertex adjacent to the root, then $v$ corresponds to 31.
    \end{itemize}
\end{lemma}
\begin{proof}
    Let $\Lambda$ be a rooted graph in $\mathcal{F}\setminus \mathcal{F}_3$. Since $\Lambda$ has a triangle, $G$ must be strictly $\PSL(2,2^5)$-solvable. By \cite[Lemma 2.1.1]{2023REU} we have that $G \cong N.M$ for some solvable $N$ and some $M\leq \Aut(\PSL(2,2^5))$ with $M \geq \operatorname{Inn}(\PSL(2,2^5)) \cong \PSL(2,2^5)$. Since 5 is not isolated, we have by \thref{noworry5} that 5 does not divide $|N|$. Thus, 5 must divide $M$, meaning $M = \Aut(\PSL(2,2^5))$. If $31$ did divide $|N|$, then letting $H$ equal a solvable strict $\{2,3,5,11\}$-subgroup of $\Aut(\PSL(2,2^5))$ via \thref{PSL232SolvableSubgroups}, the subgroup of $G$ isomorphic to $N.H$ would be a solvable group whose prime graph complement contains a triangle, a contradiction by \thref{solvclass}.
    Finally, since $\Aut(\PSL(2,2^5))$ is a section of $G$ and the edges $5-2$ and $5-3$ are not in $\pgc(\Aut(\PSL(2,2^5)))$, we have that only 11 and 31 can be adjacent to 5. It cannot be the case that $5-11 \in \pgc(G)$ and $5-31 \notin \pgc(G)$, or else we would have that $31$ divides $|N|$ by \thref{divisors}, a contradiction.
\end{proof}

\begin{proposition}\thlabel{degreeGreaterThan2PSL}
    (\ref{degreeGreaterThan2}) $\mathcal{F}_1$ is not realizable.
\end{proposition}
\begin{proof}
    Let $\Lambda$ be a rooted graph in $\mathcal{F}_1$, and suppose for contradiction that $G$ realizes it. By \thref{AutPSL232OnTop}, we have that the vertices adjacent to the root in $\Lambda$ correspond to either 11 or 31. In particular, this means there can be at most two vertices adjacent to the root in $\Lambda$, contradicting the assumption that $\Lambda \in \mathcal{F}_1$.
\end{proof}

\begin{proposition}\thlabel{noExtraHamsPSL}
    (\ref{noExtraHams}) $\mathcal{F}_2$ is not realizable.
\end{proposition}
\begin{proof}
    Suppose there exists some group $G$ that realizes a graph in $\mathcal{F}_2$.
    We can take a Hall $\{2, 3,11,31\}$-group of $G$ via \thref{HallSubgroupInducedGraph}. This group is a $\PSL(2,2^5)$-solvable group that realizes \fourVertexCompleteGraph, thus contradicting \thref{completeGraphImpossible}.
\end{proof}

\begin{proposition} \thlabel{S3PossiblePSL}
    (\ref{S3Possible}) $\mathcal{F}_3$ is realizable.
\end{proposition}
\begin{proof}
    Consider the groups from \thref{PSL2324VertexGraphs} and take their direct product with $C_5$.
\end{proof}

\begin{proposition}\thlabel{F4RealizablePSL232}
    (\ref{F4Realizable}) $\mathcal{F}_4$ is realizable.
\end{proposition}
\begin{proof}
    \phantom{.}
    \begin{itemize}
        \item The graph \fan can be realized by $\Aut(\operatorname{PSL}(2,2^5))$.
        \item The graph \northStar can be realized by $\Aut(\operatorname{PSL}(2,2^5))\times C_3$.
        \item The graph \triangleWithShortTail can be realized by $\Aut(\operatorname{PSL}(2, 2^5)) \times C_2$.
        \item The graph \codart can be realized by $\Aut(\operatorname{PSL}(2,2^5))\times C_{11}$.
        \item The graph \happyFace can be realized by $\Aut(\PSL(2,2^5)) \times C_{6}$.
        \item The graph \balloon can be realized by $(\PSL(2,2^5) \times C_{11}) \rtimes C_5$ where $C_5$ acts on $\PSL(2,2^5)$ as in $\Aut(\PSL(2,2^5))$, and Frobeniusly on $C_{11}$.
        \item The graph \dart can be realized by $\Aut(\operatorname{PSL}(2,2^5)) \ltimes (\F_2)^{20}$, via the representation of $\Aut(\operatorname{PSL}(2,2^5))$ of matching characteristic and dimension ID $a$ in \cite{ATLAS}.
        \item  The graph \cricket can be realized by $\Aut(\operatorname{PSL}(2,2^5)) \ltimes (\F_2)^{80}$, via the representation of $\Aut(\operatorname{PSL}(2,2^5))$ of matching characteristic and dimension in \cite{ATLAS}.

    \end{itemize}
\end{proof}

\begin{lemma}\thlabel{twoWithEdgePSL}
    (\ref{twoWithEdge}) Let $N$ be solvable, let $r \in \pi(N)$, and let $G \cong N.\PSL(2,2^5)$. Then $2-r \notin \pgc(G)$.

\end{lemma}
\begin{proof}
    The Sylow 2-subgroups of $\PSL(2,2^5)$ do not satisfy the Frobenius Criterion. Apply \thref{twoWithEdgeGeneral}.
\end{proof}

\begin{proposition}\thlabel{PSLtriangleWithThreeFamily}
    (\ref{triangleWithThreeFamily}) Let $\Lambda \in \mathcal{F}_5$ be a rooted graph such that there exists a triangle containing the root. Additionally, suppose that the two vertices not part of this triangle are connected by an edge. The possibilities for $\Lambda$ are exactly the following:
    \begin{center}
        \emptyhouse, \triangleWithTail, \trianglestick, \bowtiebruh, \badDart.
    \end{center}
    Then $\Lambda$ is not realizable.
\end{proposition}
\begin{proof}
    Suppose for contradiction that $\Lambda$ is realized by a $\PSL(2,2^5)$-solvable group $G$. By \thref{AutPSL232OnTop}, we have that $G \cong N.\Aut(\PSL(2,2^5))$ where $N$ is solvable, with 5 not dividing $|N|$, and additionally that the vertices of the triangle are $\{5,11,31\}$. By \thref{HallTriangles}, 11 and 31 do not divide $|N|$. Thus, $N$ is a $\{2,3\}$-group. Since $2-3 \in \Lambda$, we have by \thref{twoWithEdgePSL} that $N$ is a $2$-group. By \thref{reverseDivisors}, the subgroup of $G$ isomorphic to $N.\PSL(2,2^5)$ has prime graph complement equal to the subgraph of $\Lambda$ induced by the non-root vertices. Thus we get a contradiction by \thref{psl232ModuleExtensions}.
\end{proof}

\begin{proposition}\thlabel{PSLtriangleWithoutThreeFamily}
    (\ref{triangleWithoutThreeFamily}) Let $\Lambda \in \mathcal{F}_5$ be a rooted graph such that the root has degree 1. Additionally, suppose that there exists a triangle disjoint from the root and its neighbor. The possibilities for $\Lambda$ are exactly the following:
    \begin{center}
        \triangleIsolatedEdge, \zigZag, \hamsandwichwitharm,
    \end{center}
    Then $\Lambda$ is not realizable.
\end{proposition}
\begin{proof}
    Suppose that $\Lambda$ is realized by a $\PSL(2,2^5)$-solvable group $G$. By \thref{AutPSL232OnTop}, we have that $G \cong N.\Aut(\PSL(2,2^5))$ where $N$ is solvable, and that the vertex adjacent to the root is 31. But all triangles in $\pgc(\Aut(\PSL(2,2^5)))$ contain the vertex 31, contradicting the fact that $\pgc(G)$ has a triangle disjoint from 31.
\end{proof}

\begin{proposition}\thlabel{PSLimpossible3}
    Let $\Lambda \in \mathcal{F}_5$ be a rooted graph such that the root has degree 1. Additionally, suppose that the root's neighbor has degree 4. The possibilities for $\Lambda$ are exactly the following:
    \begin{center}
        \triangleWithTwoTails, \hamSandwichWithOtherArm.
    \end{center}
    Then $\Lambda$ is not realizable.
\end{proposition}
\begin{proof}
    Suppose that $\Lambda$ is realized by a $\PSL(2,2^5)$-solvable group $G$. By \thref{AutPSL232OnTop}, we have that $G \cong N.\Aut(\PSL(2,2^5))$ where $N$ is solvable with order not divisible by 5 or 31, and that the vertex adjacent to the root is 31. Since $11-5 \notin \pgc(G)$, $11$ divides $|N|$ by \thref{divisors}. By \thref{Nilpotent}, $N$ is nilpotent. Letting $N_{11}$ be the unique Sylow $11$-subgroup of $N$, there is a section of $G$ isomorphic to $N_{11}.\Aut(\PSL(2,2^5))$. By \thref{reverseDivisors}, the subgroup of this isomorphic to $N_{11}.\PSL(2,2^5)$ has prime graph complement equal to the subgraph of $\Lambda$ induced by the non-root vertices. Thus, we have a contradiction by \thref{psl232ModuleExtensions}.
\end{proof}

\begin{proposition} \thlabel{PSLimpossible4}
    Let $\Lambda \in \mathcal{F}_5$ be one of the following graphs:
    \begin{center}
        (1) \housenobrim, (2) \nobutt, (3) \bullcaseone, (4) \misshapenhouse.
    \end{center}
    Then $\Lambda$ is not realizable.
\end{proposition}
\begin{proof}
    Suppose that $\Lambda$ is realized by a $\PSL(2,2^5)$-solvable group $G$. By \thref{AutPSL232OnTop}, we have that $G \cong N.\Aut(\PSL(2,2^5))$ where $N$ is solvable. If $\Lambda$ is isomorphic to (1) or (2), notice that there is no way to label the graph without contradicting \thref{AutPSL232OnTop} or making it so that $\Lambda$ is not a subgraph of $\pgc(\Aut(\PSL(2, 2^5)))$. So $\Lambda$ is not isomorphic to (1) or (2). If $\Lambda$ is isomorphic to (3) or (4) it can be checked that for any labeling of the vertices of $\Lambda$ with the primes $\{2,3,5,11,31\}$ such that $\Lambda$ is a subgraph of $\pgc(\Aut(\PSL(2,2^5)))$ and $\Lambda$ follows the restrictions set by \thref{AutPSL232OnTop}, we will have $11-2 \in \Lambda$ and at least one of $11-31$ and $11-5$ is not in $\Lambda$. Thus, by \thref{divisors}, $11$ divides $|N|$. But this contradicts \thref{twoWithEdgePSL}.
\end{proof}

\begin{lemma}\thlabel{PSLBull}
     (\ref{Bull}) The graph \bull is not realizable.
\end{lemma}
\begin{proof}
    Suppose for contradiction that the above graph is realized by a $\PSL(2,2^5)$-solvable group $G$. By \thref{AutPSL232OnTop}, we have that $G \cong N.\Aut(\PSL(2,2^5))$ where $N$ is solvable, with neither 5 nor 31 dividing $|N|$. Notice that $3-31 \in \pgc(G)$ due to there being only one way to label the vertices of $\Lambda$ such that it is a rooted subgraph of $\pgc(\Aut(\PSL(2,2^5)))$.
    
    By \thref{HallTriangles}, $N$ is a $\{2,3\}$-group. By \thref{psl232ModuleExtensions}, $N$ is a strict $\{2,3\}$-group. By taking a section of the normal series given by \thref{solvableExtension}, we reduce down to two cases: 
    \begin{enumerate}
        \item $N \cong K_2.K_3$ where $K_2$ is an elementary abelian $2$-group and $K_3$ is a 3-group.
        \item $N \cong K_3.K_2$ where $K_2$ is a 2-group and $K_3$ is an elementary abelian 3-group.
    \end{enumerate}
    If (1) holds then we immediately come to a contradiction since we can take the section $K_3.\PSL(2,2^5)$ of $G$ and apply \thref{psl232ModuleExtensions} to see that the vertex $3$ must be isolated. 
    
    So assume that (2) holds. Let $G_0$ be the section of $G$ isomorphic to $K_3.K_2.\PSL(2,2^5)$. By \thref{split}, we may assume that $G_0 \cong K_3\rtimes H$, where $H \cong K_2.\PSL(2,2^5)$. We notice that $C_H(K_3)\leq K_2$ or else $C_H(K_3)K_2/K_2$ is a nontrivial normal subgroup of $H/K_2 \cong \PSL(2,2^5)$, meaning it is isomorphic to $\PSL(2,2^5)$, implying that $C_H(K_3)$ has an element of order $31$, contradicting the fact that $3-31 \in \pgc(G_0)$. Additionally, $C_H(K_3) \neq K_2$ or else $N \cong K_2 \times K_3$, so $G_0 \cong K_2.K_3.\PSL(2,2^5)$, which is a contradiction by \thref{psl232ModuleExtensions}. 
    By modding $K_2$ out by $C_H(K_3)$, we may assume without loss of generality that $K_2$ acts faithfully on $K_3$. Let $K_{31}$ be a Sylow $31$-subgroup of $H$. Since $3-31 \in \pgc(G)$, we have that the action of $K_{31}$ on $K_3$ is Frobenius and therefore faithful. We conclude that the action of $K_2K_{31}$ on $K_3$ is faithful. 

    Now we show that $K_{31}$ acts nontrivially on $K_2$ by conjugation. If the action were trivial, we would have that the action $\varphi: H/K_2 \to \Aut(Z(K_2))$ defined by $\varphi_{hK_2}(k) = h^{-1}kh$ has nontrivial kernel. But then the kernel is all of $\PSL(2,2^5)$, meaning the vertex $2$ is isolated in $\pgc(H)$, a contradiction.

    Since 31 is not a Fermat prime, by \cite[Lemma 2.1.6]{2023REU} we have that $K_{31}$ has a nontrivial fixed point in $K_3$, which contradicts the assumption that $3-31 \in \pgc(G)$.
\end{proof}

\begin{proposition}\thlabel{Realizable5VertexGraphsPSL232}
    (\ref{Realizable5VertexGraphs}) The graphs in $\mathcal{F}$ which are realizable are exactly $\mathcal{F}_3 \cup \mathcal{F}_4$.
\end{proposition}
\begin{proof}
    By \thref{degreeGreaterThan2PSL}, $\mathcal{F}_1$ is not realizable. By \thref{noExtraHamsPSL}, $\mathcal{F}_2$ is not realizable. By Propositions \ref{PSLtriangleWithThreeFamily}, \ref{PSLtriangleWithoutThreeFamily}, \ref{PSLimpossible3}, \ref{PSLimpossible4}, and \ref{PSLBull}, $\mathcal{F}_5$ is not realizable. By \thref{S3PossiblePSL}, $\mathcal{F}_3$ is realizable. By \thref{F4RealizablePSL232}, $\mathcal{F}_4$ is realizable.
\end{proof}

\subsection*{Classification}

\begin{theorem}\thlabel{PSL232Classification}
(\ref{Sz8Classification}) Given a graph $\Xi$, we have that $\Xi$ is isomorphic to the prime graph complement of some $\PSL(2,2^5)$-solvable group if and only if one of the following is true: 
\begin{enumerate}
    \item $\Xi$ is triangle-free and 3-colorable.
    \item There exists a subset $X = \{a,b,c,d\} \subseteq V(\Xi)$ equal to its own closed neighborhood in $\Xi$, where the subgraph induced by $X$ is not \fourVertexCompleteGraph, such that $X$ contains at least one triangle and $\Xi \setminus X$ is triangle-free and 3-colorable.
    \item There exists a vertex $e \in V(\Xi)$ and a subset $X = \{a,b,c,d\} \subseteq V(\Xi)$ whose closed neighborhood $N[X]$ equals $\{a,b,c,d,e\}$. Furthermore, the subgraph induced by $N[X]$ is isomorphic to one of the graph in $\mathcal{F}_4$, listed below: 
    \begin{center}
        \fan, \northStar, \triangleWithShortTail, \codart, \happyFace, \balloon, \dart, \cricket.
    \end{center}
    Additionally, $e$ must correspond to the white vertex and $\Xi \setminus X$ must be triangle-free and 3-colorable.
\end{enumerate}
\end{theorem}
\begin{proof}
    Forward direction: Let $G$ be a $\PSL(2,2^5)$-solvable group which realizes $\Xi$. If $G$ is solvable, then $\pgc(G)$ satisfies (1) by \thref{solvclass}. If $G$ is strictly $\PSL(2,2^5)$-solvable and does not satisfy (1), then $\pgc(G)[\{2,3,5,11,31\}] \in \mathcal{F}_3 \cup \mathcal{F}_4$ by \thref{sufficientOn5PSL} and \thref{Realizable5VertexGraphsPSL232}.  By \thref{disconnectpsl}, if $\pgc(G)[\{2,3,5,11,31\}] \in \mathcal{F}_3$ then $\pgc(G)$ satisfies (2), and if $\pgc(G)[\{2,3,5,11,31\}] \in \mathcal{F}_4$ then $\pgc(G)$ satisfies (3).
    
    Backward direction: We now turn to showing that given a graph $\Xi$ that satisfies (1), (2), or (3), there exists a group $G$ such that $\pgc(G)\cong\Xi$. We proceed by cases. If $\Xi$ satisfies (1), then there exists a solvable group $G$ such that $\pgc(G)\cong \Xi$ by \cite[Theorem 2.8]{2015REU}. If $\Xi$ satisfies (2), there exists a $\PSL(2,2^5)$-solvable group $H$ such that $\pgc(H)\cong \Xi[X]$ by \thref{PSL2324VertexGraphs}. By the construction given in the proof of \cite[Theorem 2.8]{2015REU}, there exists a solvable group $N$ such that $(|N|, |H|) = 1$ and $\pgc(N)\cong\Xi \setminus X$. Then we have $\pgc(H\times N)\cong \Xi$.
    
    We now turn to the final case, the case where $\Xi$ satisfies (3). By \thref{F4RealizablePSL232}, there exists a $\PSL(2,2^5)$-solvable group of the form $H \rtimes C_5$ which realizes $\Xi[\{a,b,c,d,e\}]$ such that the vertex 5 is identified with $e$, and 5 does not divide $|H|$. By the construction given in \cite[Theorem 2.8]{2015REU}, there is a solvable group of the form $N \rtimes C_5$ which realizes $\Xi\setminus X$ such that 5 is identified with $e$, and $(|N|, 3|H|) = 1$. Then $(N \times H) \rtimes C_5$ realizes $\Xi$.
\end{proof}

\section{Outlook}\label{sec:Outlook}


There are many avenues for future work on characterizing the prime graphs of different groups. In this section, we will present a few that we feel would be most natural to pursue. One possible direction for future work would be characterizing the prime graph complements of strictly $T$-solvable groups. The requirement that $T$ appear at least once in the composition series may cause certain triangle-free and 3-colorable graphs to become impossible. For example, the requirement that graphs in $\mathcal{F}$ have at least one triangle would have to be dropped in a classification of strictly $\Sz(8)$- and $\PSL(2,2^5)$-solvable groups. 

For the groups in this paper, we noticed that $\pgc(G)[\pi(T)]$ is always isomorphic to the labeled graph product of graphs given by \thref{graphExistenceFromBrauerTable}. This leads us to the following
\begin{conjecture}
Let $T$ be a nonabelian finite simple group and $k$ a prime number. For any $\chi\in\operatorname{IBr}_k(T)$ let $A_{\chi}$ be the set of edges 
$$\{k-q \mid \textnormal{there exists a } g\in T \textnormal{ such that } \operatorname{o}(g) = q \textnormal{ and } \frac{1}{\operatorname{o}(g)}\sum_{x \in \langle g \rangle}\chi(x) > 0\}.$$
Then for any strictly $T$-solvable group $G$, there exists a $Y\subseteq\bigcup_{p\in\pi(T)}\operatorname{IBr}_p(T)$ such that $$\pgc(G)[\pi(T)] = \pgc(T)\setminus\left(\bigcup_{\chi\in Y}A_\chi\right).$$ 
\end{conjecture}

Another direction for future work would be characterizing the prime graphs of groups whose composition factors are either cyclic or isomorphic to any one of the nonabelian simple groups in a set $\mathcal{T}$, a property which the authors of \cite{2022REU} call $\mathcal{T}$-solvable. This problem was solved in \cite{2022REU} for the case where $\mathcal{T}$ is comprised of all $K_3$ groups, but work has yet to be done on subsets of the $K_4$ groups. 

One could also finish the classification of $T$-solvable groups where $T$ is a $K_4$ group by classifying the remaining $\PSL(2, q)$ groups, which would involve classifying possibly infinitely many groups. As such, it would require a more general approach than what has been explored in the literature thus far.

\begin{appendices}
\section{} \label{AppendixA}
Here we list all rooted graphs on 5 vertices, derived from the list at \cite{TableOfGraphs}, which contain at least one triangle and where the root has degree 1 or 2:
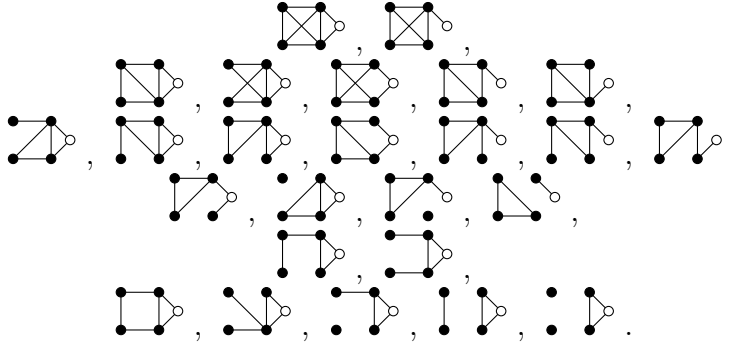
\begin{center}
    \house, \completeGraphWithArm, \\
    \fan, \nobutt, \housenobrim, \hamsandwichwitharm, \hamSandwichWithOtherArm, \\
    \dart, \badDart, \bowtiebruh, \misshapenhouse, \triangleWithTwoTails, \bullcaseone, \zigZag, \\
    \balloon, \northStar, \codart, \triangleIsolatedEdge, \\
    \triangleWithTail, \bull, \\
    \emptyhouse, \cricket, \triangleWithShortTail, \trianglestick, \happyFace.
\end{center}
\section{} \label{AppendixB}
Here we will describe the method by which we calculated the fixed points of representations via Brauer tables (and some of the more complicated ordinary character tables).\\
Whenever we evaluate the equation $$\operatorname{dim}C_V(g)=\frac{1}{o(g)}\sum_{i=1}^{o(g)}\chi(g^i),$$
for some $g\in G$ we only need to evaluate that equation for a single $g$ of order $p$ when the Sylow $p$-subgroups of $G$ are cyclic (\thref{MaslovaPaperLemma} for Brauer tables and \cite[Lemma 3.2.2]{Webb} for ordinary character tables). This is true because Sylow subgroups are conjuate, so any subgroup of order $p$ in $G$ is conjugate to all other subgroups of order $p$ in $G$. A matrix $\varphi_g$ in a representation $G\ltimes_\varphi V$ fixes a point in $V$ (and thus $\operatorname{dim}C_V(g)\neq 0$) if and only if an eigenvalue of $\varphi_g$ is 1. Because our representations (Brauer or otherwise) map $\langle g\rangle$ to a set of matrices and eigenvalues are preserved under conjugacy and the Sylow $p$-groups of $\varphi(G)$ are all conjugate, it suffices to check $\operatorname{dim}C_V(g)=\frac{1}{o(g)}\sum_{x\in\langle g\rangle}\chi(x) = \frac{1}{o(g)}\sum_{i=1}^{o(g)}\chi(g^i)$ for only one $g\in G$ of order $p$.\\
Our method for checking the dimension of the centralizer in $V$ of an element $g\in G$ of order $p$ in a representation corresponding to the character $\chi$ was to determine which conjugacy class each $x\in\langle g\rangle$ belonged to, and then add the values of $\chi$ evaluated at those conjugacy classes with the appropriate multiplicities together in GAP \cite{GAP}.

\section{} \label{AppendixC}
Any time we mention a group that can only be realized by a representation acting over some vector space we have put the matrices that the generators of the group are mapped to in our \href{https://github.com/gabriel-roca/2024-K4-Groups}{GitHub}. Many of the GAP functions we used are available on that page as well.



\end{appendices}
\section{Acknowledgements}
This research was conducted under NSF-REU grant DMS-2150205 and NSA grant H98230-24-1-0042 under the mentorship of the first author. We would like to thank Gavin Pettigrew, Saskia Solotko, and Lixin Zheng from the 2023 Texas State REU Team for their help, and the fellow students at the Texas State University 2024 REU for their support. We are also indebted to Texas State University for access to its facilities, and to the NSF for funding this research. Finally, we would like to thank Alexander Hulpke for some help with GAP.

\end{document}